\documentclass[11pt, reqno,twoside]{article}
\usepackage[english]{babel}
\usepackage[utf8]{inputenc}
\usepackage{authblk}
\usepackage{latexsym,amsfonts,amssymb,amsthm,mathrsfs,amsmath,amscd,enumerate, enumitem,esint}
\usepackage{color}
\allowdisplaybreaks
\newtheorem{teo}{Theorem}[section]
\newtheorem{cor}[teo]{Corollary}
\newtheorem{lema}[teo]{Lemma}
\newtheorem{propo}[teo]{Proposition}
\theoremstyle{definition}

\theoremstyle{remark}
\newtheorem{remark}[teo]{Remark}

\newtheorem{ej}[teo]{Example}

\theoremstyle{estilolista}
\newtheorem{itemnumerado}[equation]{}

\usepackage{titlesec}
\newpagestyle{estiloA}[]{\headrule\sethead[\thepage][L. Melchiori, G. Pradolini and W. Ramos][]{}{Commutators of potential type operators with Lipschitz symbols}{\thepage}}
\numberwithin{equation}{section}

		\def\zR{\ensuremath{\mathbb{R}}}
		\def\z{\ensuremath{\mathbb{R}^n}}
		\def\zZ{\ensuremath{\mathbb{Z}}}
		\def\zN{\ensuremath{\mathbb{N}}}
		\def\zL{\ensuremath{\mathbb{L}}}
		
		\newcommand{\ff}[3]{#1:#2\rightarrow #3}
		\newcommand{\fy}{f(y)}
		\newcommand{\fx}{f(x)}
		
		\newcommand{\lpw}[2]{L^{#1(\cdot)}_#2}
		
		\newcommand{\lp}[1]{L^{#1(\cdot)}}
		\newcommand{\lploc}[1]{\lp{#1}_{\rm{loc}}(\z)}
		\newcommand{\lplocc}[1]{L^{#1}_{\rm{loc}}(\z)}
		\newcommand{\lpr}[1]{L^{#1(\cdot)}(\z)}
		\newcommand{\normadeflpw}[3]{\left\Vert#1\right\Vert_{L^{#2(\cdot)}_#3}}
		\newcommand{\normadeflpl}[2]{\left\Vert#1\right\Vert_{#2(\cdot,L)}}
		\newcommand{\normadeflp}[2]{\left\Vert#1\right\Vert_{L^{#2(\cdot)}}}

		\newcommand{\normadefp}[2]{\left\Vert#1\right\Vert_{#2}}
		\newcommand{\normadeflplogq}[3]{\left\Vert#1\right\Vert_{L^{#2(\cdot)}(\log L)^{#3(\cdot)}}}
		\newcommand{\esp}{\,\,\,\,\,\,\,\,\,}
		\newcommand{\texto}[1]{\textrm{#1}}
		\newcommand{\ex}[1]{#1(\cdot)}
		\newcommand{\ca}{\mathcal{X}_Q}
		
		\newcommand{\car}{\mathcal{X}}
		
		\newcommand{\inplogdern}{\in\mathcal{P}^{log}(\z)}

		\newcommand{\mbetas}[2]{M_{#1(\cdot),#2}}
		
		\newcommand{\mbetasf}[3]{M_{#1(\cdot),#2}#3}

		\newcommand{\tfi}{T_K}

\usepackage[colorlinks,linkcolor = blue,citecolor = blue]{hyperref}
\begin{document}
	\title{Commutators of potential type operators with Lipschitz symbols on variable Lebesgue spaces with different weights}
		
	\author{Luciana Melchiori\thanks{lmelchiori@santafe-conicet.gov.ar, CONICET-UNL, Santa Fe, Argentina.}, Gladis Pradolini\thanks{gpradolini@santafe-conicet.gov.ar, CONICET-UNL, Santa Fe, Argentina.}\, and Wilfredo Ramos\thanks{oderfliw769@gmail.com. CONICET-UNNE, Corrientes, Argentina.}}
	
	\renewcommand{\thefootnote}{\fnsymbol{footnote}}
	\date{\vspace{-1.5cm}}
	
	\maketitle
	
	\begin{abstract}
		We prove that a generalized Fefferman-Phong type condition on a pair of weights $u$ and $v$ is sufficient for the boundedness of the commutators of potential type operators from $\lpw{p}{v}$ into $\lpw{q}{u}$.
	We also give an improvement of this result in the sense that we not only consider a variable version of power bump conditions, but also weaker norms related to Musielak-Orlicz functions.
	
	We consider a wider class of symbols including Lipschitz symbols and some generalizations.
	\end{abstract}

\section{Introduction and main results}\label{S1}

	In \cite{SW}, E. Sawyer and R. Wheeden obtained Fefferman-Phong type conditions on a pair of weights in order to prove boundedness results for the fractional integral operator $I_\alpha$, between Lebesgue spaces with different weights.
	For the case of one weight, remarkably simple conditions on the weight characterizing the boundedness of $I_\alpha$ were known to hold (see \cite{MuckenhouptWheeden}).
	Motivated by the results above, in \cite{P}, C. Pérez considered weaker norms than those involved in the Fefferman-Phong type conditions in \cite{SW}, and obtained two-weighted boundedness estimates for the potential operator $\tfi$, formally defined by
	$$\tfi\fx=\int_{\z} K(x-y)\fy \, dy,$$
	whenever this integral is finite where the kernel $K$ is a non-negative and locally integrable function satisfying certain weak growth condition.
	This article was the motivation for a great variety of subsequent papers related to this kind of operator.
	For example, in \cite{LI2} and \cite{LQY}, the authors obtained weighted $L^p$ inequalities of Fefferman-Stein type for $\tfi$ and for the higher order commutators with BMO symbols associated to this operator, respectively, whenever $1<p<\infty$.
	If $b\in\lplocc{1}$ and $m\in\zN$, the commutator of order $m$ of $\tfi$ is formally defined by
	\begin{equation}\label{conmutador}
	\tfi^{b,m} \fx = \int_{\zR^n} (b(x)-b(y))^m K(x-y) f(y)dy,
	\end{equation}
	whenever this integral is finite.
	In the multilinear context, similar results were proved in \cite{BGP}.
	For these commutators two-weighted norm inequalities in the spirit of those given in \cite{P} were proved in \cite{Li} in the classical $L^p$ context, and in \cite{MP2} on the general setting of variable Lebesgue spaces.\\
	
	The commutators of fractional type operators with Lipschitz symbols were studied by several authors. For instance, in \cite{MY} the authors considered unweighted estimates for the mentioned operator acting between different Lebesgue spaces in the context of non-doubling measures.
	
	Characterizations of Lipschitz functions via the boundedness of commutators of fractional integral operators with generalized Lipschitz symbols were given in \cite{PRa}, in the general setting of variable Lebesgue spaces.
	
	In \cite{DPR} the authors give weighted $L^p-L^q$ estimates for the commutators, with Lipschitz symbols, of a great variety of fractional type operators. Later, in \cite{P} certain extrapolation techniques allow to obtain similar results in variable Lebesgue spaces.\\
	
	The main aim of this paper is to describe the behavior of the commutators of the potential type operators $\tfi^{b,m}$ between variable Lebesgue spaces with different weights, for a wider class of symbols $b$ including Lipschitz symbols and some generalizations.
	Concretely, we prove that a generalized Fefferman-Phong type condition on a pair of weights $u$ and $v$ is sufficient for the boundedness of the commutator $\tfi^{b,m}$, from $\lpw{p}{v}$ into $\lpw{q}{u}$.
	
	When the symbol $b$ belongs to a variable Lipschitz space, we not only consider variable version of power bump conditions, but also we consider weaker norms related to Musielak-Orlicz functions.
	Thus, in this sense, we are providing an improvement.\\
	
	In the definition of $\tfi^{b,m}$, the function $K$ belongs to a certain class of kernels that satisfy that there exists positive constants $\delta$, $c$ and $0\leq\varepsilon<1,$ with the property that
	\begin{equation*}\label{condicion en el nucleo}
	\sup_{2^k<|x|\leq2^{k+1}} K(x)\leq\frac{c}{2^{kn}}\int_{\delta(1-\varepsilon)2^k<|y|\leq2\delta(1+\varepsilon)2^k} K(y)\,dy,
	\end{equation*}
	for all  $k\in\zZ.$ We shall denote this class by  $\mathfrak{D}$.
	
	For example, if $K$ is radial an non-increasing, then $K\in\mathfrak{D}$. 
	A basic example of potential operator with radial and non-increasing kernel $K$ is given by the fractional integral operator $I_\alpha$, which is the convolution with the kernel $K(t)=|t|^{\alpha-n}$, $0<\alpha<n.$ 
	There are other important examples such as the Bessel potential $J_{\beta,\lambda}$, $\beta\,,\lambda>0$ with kernels $K_{\beta,\lambda}$ best defined by means of its Fourier transform by $\widehat{K_{\beta,\lambda}}(\xi)=(\lambda^2+|\xi|^2)^{-\beta/2}$
	and $K_{\beta,\lambda}$ is also radial and non-increasing.
	
	Nevertheless, condition $\mathfrak{D}$ involves other type of kernels $K$ such as radial and non-decreasing functions. Moreover, if $K$ is essentially constant on annuli, that is, $K(y)\leq CK(x)$ for $|y|/2\leq|x|\leq2|y|$, then $K\in\mathfrak{D}$.\\
	
	We will be working in a general context that we now introduce.
	
	Let $\ff{\ex{p}}{\z}{[1,\infty]}$ be a measurable function. For $\mathrm{A}\subset\z$ we define
	$$p^-_\mathrm{A}= \inf_{x\in \mathrm{A}}p(x)\esp\esp\esp p^+_\mathrm{A}= \sup_{x\in \mathrm{A}}p(x).$$
	For simplicity we denote $p^-=p^-_{\z}$ and $p^+=p^+_{\z}$.
	
	With $\ex{p'}$ we denote the conjugate exponent of $\ex{p}$ given by $\ex{p'}=\ex{p}/(\ex{p}-1)$.
	It is not hard to prove that $(p')^-=(p^+)'$ and $(p')^+=(p^-)'$.
	
	We say that $\ex{\alpha}:\z\rightarrow\zR$ is {\em globally log-H\"older continuous} on $\z$ if it satisfies the following inequalities
	\begin{equation*}\label{log1}
	\left|\alpha(x)-\alpha(y)\right|\leq \frac{C}{\log(e+1/|x-y|)}, \,\,x,y\in\z
	\end{equation*}
	and
	\begin{equation}\label{log2}
	\left|\alpha(x)-\alpha_\infty\right|\leq \frac{C}{\log(e+|x|)}, \,\, x\in\z
	\end{equation}
	for some positive constants $C$ and $\alpha_\infty$.
	It is easy to see that the inequality (\ref{log2}) implies that $\lim_{|x|\rightarrow \infty} \alpha(x) = \alpha_\infty$.
	
	We say that $\ex{p}\in\mathcal{P}(\z)$ if $1\le p^-\leq p^+\le\infty$ and we denote by $\mathcal{P}^{{\rm log}}(\z)$ the set of the exponents $\ex{p}\in\mathcal{P}(\z)$ such that $1/\ex{p}$ is globally log-H\"older continuous.    
	If $p\in\mathcal{P}(\z)$ with $p^+<\infty$, then $p\in\mathcal{P}^{{\rm log}}(\z)$ if and only if $p$ is globally log-H\"older continuous.
	
	If $\ex{p}\in\mathcal{P}(\z)$, we define the function
	$$\varphi_{\ex{p}}(y,t)=
	\left\{
	\begin{array}{lc}
	t^{p(y)}, & 1\le p(y) <\infty \\
	\infty\cdot\chi_{(1,\infty)}(t), & p(y)=\infty,
	\end{array}
	\right.$$
	for $t\ge0$ and $y\in\z$, with the convention $\infty\cdot 0=0$, where $\car_{(1,\infty)}$ denote the characteristic function of $(1,\infty)$. 
	Then the variable exponent Lebesgue space $\lpr{p}$ is the set of the measurable functions $f$ defined on $\z$ such that, for some positive $\lambda$,
	$$\int_{\z} \varphi_{\ex{p}}(x,|\fx|/\lambda)\, dx<\infty.$$
	A Luxemburg norm can be defined in $\lpr{p}$ by taking
	$$\normadeflp{f}{p}=\inf\left\lbrace \lambda >0 : \int_{\z} \varphi_{\ex{p}}(x,|\fx|/\lambda)\, dx\leq 1 \right\rbrace.$$    
	By $\lploc{p}$ we denote the space of the functions $f$ such that $f\chi_U\in\lpr{p}$ for every compact set $U\subset\z$.
	
	A locally integrable function $w$ defined in $\z$ which is positive almost everywhere is called a weight.
	For $\ex{p}\in\mathcal{P}(\z)$ we define the weighted variable Lebesgue space $\lpw{p}{w}(\z)$ as the set of the measurable functions $f$ defined on $\z$ such that $fw\in\lp{p}(\z)$, and $\normadeflpw{f}{p}{w}:=\normadeflp{fw}{p}$.
	
	By a cube $Q$ in $\z$ we shall understand a cube with sides parallel to the coordinate axes. The sidelength of $Q$ is denoted by $\ell(Q)$ and $\gamma Q$, $\gamma >0$, denotes the cube concentric with $Q$ and with sidelength $\gamma\ell (Q)$. 
	
	We shall say that $A\simeq B$ if there exist two positive constants $C_1$ and $C_2$ such that $C_1B\leq A\leq C_2 B$.\\

	We define now the functional related with the space where the symbol $b$ belongs. 
	We use $\mathcal{E}$ to denote the class of all cubes $Q$ in $\z$ with sides parallel to the axes and consider a functional $a :\mathcal{E} \rightarrow[0,\infty)$.
	We say that $a$ satisfies the $T_\infty$ condition and we denote by $a\in T_\infty$, if there exists a finite positive constant $t_\infty$ such that for every $Q,Q'\in\mathcal{E}$ such that $Q'\subset Q$,
	\begin{equation}\label{cond t infinito}
	a(Q') \leq t_\infty\, a(Q).
	\end{equation}
	We denote the least constant $t_\infty$ in (\ref{cond t infinito}) by $\|a\|_{t_\infty}$. Clearly, $\|a\|_{t_\infty}\geq 1$.
	
	Let $0 < \varrho < \infty$ and $a\in T_\infty$.
	We say that a function $b\in\lplocc{1}$ belongs to the generalized Lipschitz space $\mathcal{L}^\varrho_a$ if
	\begin{equation}\label{lipschitz a}
	\sup_Q \frac{1}{a(Q)}\left(\frac{1}{|Q|}\int_Q |b-b_Q|^\varrho\,dx\right)^{1/\varrho}<\infty
	\end{equation}
	where the supremum is taken over all cubes $Q\subset\z$ and $b_Q$ denote the average $\frac{1}{|Q|}\int_Q b$ (which sometimes will be denoted by $\fint_Q b$).\\
	
	We are now in position to state our main results.
	
	The next theorem gives a two weighted boundedness result between variable Lebesgue spaces with different exponents for the commutator $\tfi^{b,m}$, when the symbol $b$ belongs to the class $\mathcal{L}^\varrho_a$ defined previously.
	The function $\widetilde{K}$ involved in the condition on the weights is given by
	$$\widetilde{K}(t)=\int_{|z|\leq t} K(z)\,dz.$$
	\\
	
	\begin{teo}\label{teo conm variable con lipschitz generalizado}
		Let  $\ex{p},\ex{q}\inplogdern$ such that $1<p^-\le\ex{p}\leq\ex{q}\le q^+<\infty$, $K\in\mathfrak{D}$ and $m\in\zN \cup\{0\}$.
		Let $1\leq \varrho < \infty$, $a\in T_\infty$ and $b\in\mathcal{L}^\varrho_a$.
		Let $R,S$ be two constants such that $R>(p')^+/(p')^-$ and $S>q^+/q^-$.
		Suppose that $(v,w)$ is any couple of weights such that $v\in \lploc{p}$, $w\in\lploc{Sq}$ 
		and, for some positive constant $\kappa$ and for every cube $Q$,  
		\begin{equation}\label{1.z}
		a(Q)^{m}
		\widetilde{K}(\ell(Q))
		\frac{\normadeflp{\car_Q}{q}}{\normadeflp{\car_Q}{p}}
		\frac{\normadeflp{\car_Qv^{-1}}{Rp'}}
		{\normadeflp{\car_Q}{Rp'}}
		\frac{\normadeflp{\car_Qw}{Sq}}
		{\normadeflp{\car_Q}{Sq}}
		\le \kappa.
		\end{equation}
		Then $$\tfi^{b,m}:\lpw{p}{v}(\z) \hookrightarrow \lpw{q}{w}(\z).$$
		More precisely, 
		$$\normadeflpw{\tfi^{b,m}f}{q}{w}\lesssim\kappa\normadefp{b}{{\mathcal{L}^1_a}}^m\normadeflpw{f}{p}{v}, \forall f\in\lpw{p}{v}(\z).$$\\
	\end{teo}
	In the classical Lebesgue spaces, a proof can be found in \cite{P} for the case $m=0$, that is, $\tfi^{b,m}=\tfi$; and in \cite{Li} for $m\geq1$ and $b\in BMO=\mathcal{L}^1_a$ where $a(Q)=1$.
	In the variable Lebesgue spaces, when $b\in BMO$ the result above was proved in \cite{MP2}.
	
	Let us observe that, if $a(Q) = |Q|^{\delta/n}$, $0 < \delta < 1$, then $a\in T_\infty$ and it is known that $\mathcal{L}^1_a:=\zL(\delta)$ coincides with the classical Lipschitz spaces $\Lambda_\delta$ define as the set of functions $b$ such that
	$$|b(x)-b(y)|\leq C |x-y|^\delta$$
	for some positive constant $C$ and for every $x,y\in\z$.
	
	On the other hand, if $\ex{r}\inplogdern$ with $r^+<\infty$,
	\begin{equation}\label{delta}
	1<\gamma\leq r^-\leq r^+< \frac{n\gamma}{(n-\gamma)^+} \,\,\,\texto{and}\,\,\, 0\leq\ex{\delta}/n:=1/\gamma-1/r(\cdot)<1/n,
	\end{equation}
	in [\cite{PRa}, Corollary 3.6] it was proved that the functional $a(Q)=|Q|^{1/\gamma-1}\normadeflp{\car_Q}{r'}$ satisfies the $T_\infty$ condition and $\mathcal{L}^1_a=\zL(\ex{\delta})$ are a variable version of the spaces $\zL(\delta)$ defined above.
	Indeed, let us observe that the functional above can be written as
	\begin{align*}
	a(Q)\simeq
	|Q|^{1/\gamma}/\normadeflp{\car_Q}{r}
	\simeq\normadeflp{\car_Q}{n/\delta},
	\end{align*}
	(see Lemmas \ref{p en plog} and \ref{equivalenciabeta}).\\
	
	In the case that $b\in\zL(\ex{\delta})$, we can improve the theorem above in the sense that we can introduce other type of norms in the conditions on the weights involving generalized $\Phi$-functions (G$\Phi$-functions) (see Section \ref{S2} for more information about G$\Phi$-functions).
	In order to state the results we need some definitions.
	
	The norm associated to a given G$\Phi$-function $\Psi$ is given by
	$$\normadeflpl{f}{\Psi}=\inf\left\{\lambda>0:\int_{\z}\Psi\left(x,\frac{|f(x)|}{\lambda}\right)\,dx\leq 1\right\}$$
	and we denote by $L^\Psi(\z)$ the space of functions $f$ such that $\normadeflpl{f}{\Psi}<\infty$.
	
	A corresponding maximal operator associated to $\Psi$ is
	\begin{equation}\label{Mvar}
	M_{\Psi(\cdot,L)} f(x)=\sup_{Q\ni x}\frac{\normadeflpl{\ca f}{\Psi}}{\normadeflpl{\ca}{\Psi}}
	\end{equation}
	and, for $\ex{\beta}\in\mathcal{P}(\z)$, we define the following fractional type version of maximal above as follows
	\begin{equation}\label{Mvarfrac}
	M_{\ex{\beta},\Psi(\cdot,L)} f(x)=\sup_{Q\ni x}\normadeflp{\ca}{\beta}\frac{\normadeflpl{\ca f}{\Psi}}{\normadeflpl{\ca}{\Psi}}.
	\end{equation}
	
	We say that a 3-tuples of G$\Phi$-functions $(A,B,D)$ satisfy condition $\mathcal{F}$ if they verify
	
	\begin{itemnumerado}\label{desigualdad normas}
		$\normadefp{\ca}{A(\cdot,L)}\normadefp{\ca}{B(\cdot,L)}\lesssim\normadefp{\ca}{D(\cdot,L)}$ where $\lesssim$ means that there exists a positive constant $C$ such that \textit{\ref{desigualdad normas}} holds with $\lesssim$ replaced by $\leq C$.
	\end{itemnumerado}
	\begin{itemnumerado}\label{a b y c para holder}
		$A^{-1}(x,t)B^{-1}(x,t)\lesssim D^{-1}(x,t)$ where $A^{-1}$ denote the inverse of $A$ (for the definition of the inverse of a G$\Phi$-function see Section \ref{S2}).
	\end{itemnumerado}
	\begin{itemnumerado}
		$\normadeflpl{\ca}{D}\normadeflpl{\ca}{D^*}\lesssim |Q|$, where $D^*$ is the conjugate function of $D$ (for the definition of the conjugate of a G$\Phi$-function see Section \ref{S2}).\label{D y D conjugada}
	\end{itemnumerado}
	
	Necessary conditions on $D$ where given in [\cite{DHHR}, Remark 4.5.8] and [\cite{HH}, Lemma 4.4.5.] in order to verify \textit{\ref{D y D conjugada}}.
	
	We shall give later some examples of G$\Phi$-functions that satisfy condition $\mathcal{F}$.
	
	We can now give our result.
	
	\begin{teo}\label{teo general conmut lipchitz constante}
		Let  $\ex{p},\ex{q}\inplogdern$ such that $\ex{p}\leq\ex{q}$, $K\in\mathfrak{D}$, $m\in\zN\cup\{0\}$.
		Let $\ex{\beta}$ be a function such that $1/\ex{\beta}=1/\ex{p}-1/\ex{q}$.
		Let $\ex{r}\inplogdern$ and $\ex{\delta}$ defined as in (\ref{delta}), such that $r_\infty\leq\ex{r}$ and let $b\in\zL(\ex{\delta})$.
		Let $(A,B,D)$ and $(E,H,J)$ G$\Phi$-functions satisfying condition $\mathcal{F}$,
		\begin{equation}\label{hipotesis maximal}
		M_{B(\cdot,L)}:L^{\ex{p}}(\z)\rightarrow L^{\ex{p}}(\z)
		\end{equation}
		and
		\begin{equation}\label{hipotesis maximal fraccionaria}
		M_{\ex{\beta},H(\cdot,L)}:L^{\ex{q'}}(\z)\rightarrow L^{\ex{p'}}(\z).
		\end{equation}
		Suppose that $(v,w)$ is any couple of weights such that $v\in\lploc{p}$ and, for some positive constant $\kappa$ and for every cube $Q$,
		\begin{equation}\label{hipotesis FM MO}
		\normadeflp{\car_Q}{n/\delta}^{m}
		\widetilde{K}(\ell(Q))
		\frac{\normadeflp{\car_Q}{q}}{\normadeflp{\car_Q}{p}}
		\frac{\normadefp{\car_Qv^{-1}}{A(\cdot,L)}}
		{\normadefp{\car_Q}{A(\cdot,L)}}
		\frac{\normadefp{\car_Qw}{E(\cdot,L)}}
		{\normadefp{\car_Q}{E(\cdot,L)}}
		\le \kappa.
		\end{equation}
		Then $$\tfi^{b,m}:\lpw{p}{v}(\z) \hookrightarrow\lpw{q}{w}(\z).$$                More precisely, 
		$$\normadeflpw{\tfi^{b,m}f}{q}{w}\lesssim\kappa\normadeflpw{f}{p}{v}, \forall f\in\lpw{p}{v}(\z).$$
		
	\end{teo}

	Let us give some examples of G$\Phi$-functions that satisfy the hypothesis of the theorem above.
	
	Notice first that, if we consider $\ex{p}\in\mathcal{P}(\z)$ and $\ex{q}$ with $q^+<\infty$, then
	$\Psi(x,t)=t^{p(x)}(\log(e+t))^{q(x)}$, $x\in\z,$ $t\ge0$,
	is a G$\Phi$-function.
	In this case, the space $L^\Psi(\z)$ will be denoted by $L^{\ex{p}}(\log L)^{\ex{q}}(\z)$.
	In [\cite{MMO}, Proposition 2.5] the authors proved that the Hardy-Littlewood maximal operator $M$ is bounded in this space when $\ex{p}\inplogdern$ with $1<p^-\le p^+<\infty$, and $\ex{q}\in\mathcal{P}^{{\rm loglog}}(\z)$.	
	We say that $\ex{q}\in\mathcal{P}^{{\rm loglog}}(\z)$ if $\ff{\ex{q}}{\z}{\zR}$ with $q^+<\infty$ such that, for some positive constant $C$, it satisfies the following inequality
	\begin{equation*}\label{loglog}
	|q(x)- q(y)|\leq \frac{C}{\log(e+\log(e+1/|x-y|))}, \textrm{ for every } x,y\in\z.
	\end{equation*}
	\medspace	
	\begin{remark}\label{plog y ploglog}
		Note that if $\ex{p}\inplogdern$ with $p^+<\infty$ and $\ex{q}\in\mathcal{P}^{{\rm loglog}}(\z)$, then $\ex{(pq)}\in\mathcal{P}^{{\rm loglog}}(\z)$ and $\ex{(q/p)}\in\mathcal{P}^{{\rm loglog}}(\z)$.
		Indeed, for every $x,y\in\z$,
		\begin{align*}
		|p(x)q(x)-p(y)q(y)|
		&\leq |p(x)||q(x)-q(y)|+|q(y)||p(x)-p(y)|\\
		&\lesssim \frac{p^+}{\log(e+\log(e+1/|x-y|))} + \frac{q^+}{\log(e+1/|x-y|)}\\
		&\lesssim \frac{1}{\log(e+\log(e+1/|x-y|))}.
		\end{align*}
		This gives $\ex{(pq)}\in\mathcal{P}^{{\rm loglog}}(\z)$. 
		Since $(1/p)^+<\infty$, $\ex{(q/p)}\in\mathcal{P}^{{\rm loglog}}(\z)$ follows from the first property.
	\end{remark}
	\medspace
	\textit{Examples.}
	Let $\ex{p}\inplogdern$ with $1<p^-\le p^+<\infty$ and $\sigma>(p')^+/(p')^-$.
	The following G$\Phi$-functions satisfy condition $\mathcal{F}$ and the hyphoteses \eqref{hipotesis maximal} and \eqref{hipotesis maximal fraccionaria} of the Theorem \ref{teo general conmut lipchitz constante}.
	\begin{ej}\label{ejemplo 1}
		$A_1(x,t)=t^{\sigma p'(x)}(\log(e+t))^{\sigma p'(x)}$,
		$B_1(x,t)=t^{(\sigma p')'(x)}$ and
		$D_1(t)=t \log(e+t)$.
	\end{ej}
	\begin{ej}\label{ejemplo 2}
		If, in addition, $\ex{\mu}\inplogdern$ with $1<\mu^-\le\mu^+<\infty$ such that 
		\begin{equation*}
		1/\sigma\ex{p'}-1/\ex{\mu}>\epsilon
		\end{equation*}
		for some constant $\epsilon\in(0,1)$ and $\ex{\nu}\in\mathcal{P}^{{\rm loglog}}(\z)$ then, the example is given by
		$A_2(x,t)=t^{\mu(x)}(\log(e+t))^{\nu(x)\mu(x)}$, $B_2(x,t)=t^{(\sigma p')'(x)}$ and $D_2(x,t)=t^{\alpha(x)} (\log(e+t))^{\alpha(x)\nu(x)}$ where $\ex{\alpha}$ is defined by $1/\ex{\alpha}=1/\ex{\mu}+1/\ex{(\sigma p')'}.$
	\end{ej}
	In Section $\ref{S3}$ we check these examples.
	\bigskip
	
	The paper is organized as follows.
	In Section \ref{S2} we introduce basic definitions and known results related to Musielak-Orlicz spaces.
	We also give some boundedness estimates in this context.
	In Section \ref{S3} we prove a key estimate regarding the $\lp{p}(\log L)^{\ex{q}}(\z)$ norm of $\ca$ for $Q\in\mathcal{E}$, using a series of auxiliary lemmas that we prove as well. 
	We also discuss the validity of Examples \ref{ejemplo 1} and \ref{ejemplo 2}.
	Finally, in Section \ref{S4} we prove Theorem 1.1 and Theorem 1.2.

\section{Preliminaries}\label{S2}
	In this section we give some previous definitions and results that we shall be using throughout this paper.
	
	With $\mathcal{M}$ we denote the set of all Lebesgue real valued, measurable functions on $\z$.
	
	A convex function $\psi:[0,\infty)\rightarrow[0,\infty]$ with
	$\psi(0) = 0$,
	$\lim_{t\rightarrow 0^+}\psi(t) = 0$ and $\lim_{t\rightarrow \infty}\psi(t) =\infty$ is called a $\Phi$-function.
	
	A real function $\Psi:\z\times[0,\infty)\rightarrow[0,\infty]$ is said to be a generalized $\Phi$-function (G$\Phi$-function), and we denote $\Psi\in G\Phi(\z)$,
	if $\Psi(x,t)$ is Lebesgue-measurable in $x$ for every $t\geq0$ and $\Psi(x,\cdot)$ is a $\Phi$-function
	for every $x\in\z$.
	
	If $\Psi\in G\Phi(\z)$, then the set
	$$L^\Psi(\z):=\left\{ f\in\mathcal{M}: \int_{\z} \Psi\left(x,|f(x)|\right)\,dx<\infty\right\}$$
	defines a Banach function space equipped with the Luxemburg-norm given by
	$$\normadeflpl{f}{\Psi}:=\inf\left\{ \lambda>0: \int_{\z} \Psi\left(x,\frac{|f(x)|}{\lambda}\right)\,dx\leq1\right\}.$$
	The space $L^\Psi(\z)$ is called a Musielak-Orlicz space.
	
	Let $\ex{p}\in\mathcal{P}(\z)$, then $\Psi(x,t)=t^{p(x)}\in G\Phi(\z)$.
	In this case, the space $L^\Psi(\z)$ is the variable exponent Lebesgue space $L^{\ex{p}}(\z)$ defined in the introduction.
	If we also consider $\ex{r}$ with $r^+<\infty$, then $\Psi(x,t)=t^{p(x)}(\log(e+t))^{r(x)}\in G\Phi(\z)$.
	In this case, the space $L^\Psi(\z)$ is the space $\lp{p}(\log L)^{r(\cdot)}(\z)$ introduced before.
	
	Let $\Psi\in G\Phi(\z)$, then for any $x\in\z$ we denote by $\Psi^*(x,\cdot)$ the conjugate function of $\Psi(x,\cdot)$ which is defined by
	$$\Psi^*(x,u)=\sup_{t\geq0}(tu-\Psi(x,t)),\esp  u\geq 0.$$
	For $\Psi\in G\Phi(\z)$ that verifies that every simple function belongs to $L^{\Psi^*}(\z)$,  we have the following norm conjugate formula,
	\begin{equation}\label{dual}
	\normadeflpl{f}{\Psi}\simeq \sup_{\normadeflpl{g}{\Psi^*}\leq1}\int_{\z} |f(x)g(x)|\,dx
	\end{equation}
	for every function $f\in L^{\Psi}(\z)$ (see [\cite{DHHR}, Corollary 2.7.5]).
	
	The following lemma can be deduced from Lemma 4.4.5 in \cite{HH}.
	\begin{lema}\label{HH 4.4.5 }
		Let $\psi$ a $\Phi$-function, then the following inequality 
		$$\normadefp{\ca}{\psi}\normadefp{\ca}{\psi^*}\lesssim |Q|$$
		holds for every cube $Q$ in $\z$.		 
	\end{lema}	
	Also we can define $\Psi^{-1}$, the generalized inverse function of $\Psi$, by
	$$\Psi^{-1}(x,t):=\inf\{u\geq0: \Psi(x,u)\ge t\},\esp  x\in\z, t\geq 0.$$	
	For example, if $p\in\mathcal{P}(\z)$ and $\Psi(x,t)=t^{p(x)}$, $\Psi^{-1}(x,t)=t^{1/p(x)}$ and $\Psi^*(x,t)=t^{p'(x)}$.
	
	Note that, by definition of $\Psi^*$,  the following generalization of the Young's inequality holds in this context,
	\begin{equation}\label{young}
	vu\leq \Psi(x,v)+\Psi^*(x,u),\esp \forall x\in\z, \forall v,u\geq 0,
	\end{equation}
	for any $\Psi\in G\Phi(\z)$.
	If we put $v=\Psi^{-1}(x,t)$ and $u=(\Psi^*)^{-1}(x,t)$ in equation \eqref{young} we obtain 
	\begin{equation}\label{inversas y conjugadas}
	\Psi^{-1}(x,t)(\Psi^*)^{-1}(x,t)
	\leq \Psi(x,\Psi^{-1}(x,t))+\Psi^*(x,(\Psi^*)^{-1}(x,t))
	\le 2t.
	\end{equation}
	Moreover, it can be proved that if $\Psi,\Lambda,\Theta\in G\Phi(\z)$ such that $\Psi(x,\cdot),\Lambda(x,\cdot)$ are strictly increasing and 
	$\Psi^{-1}(x,t)\Lambda^{-1}(x,t)\leq\Theta^{-1}(x,t)$ for every $x\in\z$, and for every $t\geq 0$,
	then
	$$\Theta(x,tu)\leq \Psi(x,t)+\Lambda(x,u),\esp  \forall x\in\z, \forall t,u\geq 0.$$
	The inequality above allows us to prove the following generalized H\"older type inequality in this context.
	The proof is standard and we omit it.
	
	\begin{lema}
		Let $\Psi,\Lambda,\Theta\in G\Phi(\z)$ such that $\Psi(x,\cdot),\Lambda(x,\cdot)$ are strictly increasing and
		$$\Psi^{-1}(x,t)\Lambda^{-1}(x,t)\leq\Theta^{-1}(x,t),\esp  \forall x\in\z, \forall t\geq 0.$$
		Then
		\begin{equation}\label{holdermusie}
		\normadeflpl{fg}{\Theta} \lesssim\normadeflpl{f}{\Psi}\normadeflpl{g}{\Lambda}	
		\end{equation}for all $f\in L^\Psi(\z)$ and $g\in L^{\Lambda}(\z)$.
	\end{lema}
	
	For example, if $\Theta(x,t)=t^{s(x)}$, $\Psi(x,t)=t^{p(x)}$ and $\Lambda(x,t)=t^{q(x)}$ with $\ex{s},\ex{p},\ex{q}\in\mathcal{P}(\z)$ and $1/\ex{s}=1/\ex{p}+1/\ex{q}$, we obtain that
	\begin{equation}\label{holderrpq}
	\normadeflp{fg}{s}\lesssim\normadeflp{f}{p}\normadeflp{g}{q}.
	\end{equation}
	In the case $\ex{s}\equiv1$ inequality \eqref{holderrpq} becomes
	\begin{equation}\label{holderpp}
	\int_{\z}|f(y)g(y)|\,dy\lesssim\normadeflp{f}{p}\normadeflp{g}{p'}
	\end{equation}
	and, for a general $\Psi\in G\Phi(\z)$ such that $\Psi(x,\cdot)$ is strictly increasing, from inequality \eqref{inversas y conjugadas} we obtain
	\begin{equation}\label{holder fi y fi conjugada}
	\int_{\z}|f(y)g(y)|\,dy
	\lesssim\normadeflpl{f}{\Psi}\normadeflpl{g}{\Psi^*},
	\end{equation} 
	which is an extension of the classical H\"{o}lder inequality (see \cite{DHHR}). \\

	Particularly, when we deal with variable Lebesgue spaces, we have the following known results that we shall be using along this paper.
	
	\begin{lema}[\cite{DHHR}, Lemma 3.4.2]\label{modular y norma}
		Let  $\ex{p}\in \mathcal{P}(\z)$ with $p^+<\infty$.
		Then
		$$\normadeflp{f}{p}\leq C_1 
		\esp\texto{ if and only if }\esp 
		\int_{\z} |f(x)|^{p(x)}\,dx\leq C_2.$$
		Moreover, if either constant equals 1 we can take the other equal to 1 as well.
	\end{lema}
	\medskip
	
	The following lemma describes some properties of the exponent in $\mathcal{P}^{\rm log}(\z)$.
	\begin{lema}\label{propiedades plog}
		Let $\ex{p},\ex{q}\inplogdern$ and $c\in\zR$ such that $c\geq1/p^-$, then the following properties hold: \\		
		\rm (i) $\ex{cp}\inplogdern$.\\		
		\rm (ii) $\ex{p'}\inplogdern$.\\
		\rm (iii) If $\ex{\alpha}$ is the exponent defined by $1/\ex{\alpha}=1/\ex{p}+1/\ex{q}$, then $\ex{\alpha}\inplogdern$.\\		
		\rm (iv) If, in addition, $p^+,q^+<\infty$, then $\ex{(pq)}\inplogdern$.
	\end{lema}
	
	The next lemma can be deduced from the Corollary 4.5.9 in \cite{DHHR}.
	\begin{lema}[\cite{DHHR}]\label{p en plog}
		Let  $\ex{p}\inplogdern$.
		Then there exists two positive constants $C_p^*$ and $C_p^{**}$ such that 
		$$|Q|\le C_p^*\normadeflp{\car_Q}{p}\normadeflp{\car_Q}{p'} \leq C_p^{**}|Q|,$$
		for every cubes $Q\subset\z$. Note that we can suppose $C_p^*,C_p^{**}\ge1$.
	\end{lema}
	Moreover, we have the following result.
	\begin{lema}[\cite{MP2}, Lemma 2.7]\label{equivalenciabeta}
		Let $\ex{p}, \ex{q} \inplogdern$ such that $\ex{p}\leq \ex{q}$.
		Suppose that $1/\beta(\cdot)=1/\ex{p}-1/\ex{q}$, then
		$$\normadeflp{\car_Q}{p}\normadeflp{\car_Q}{q}^{-1} \simeq \normadeflp{\car_Q}{\beta},$$
		for every cube $Q\subset\z$.
	\end{lema}
	Note that Lemma \ref{equivalenciabeta} implies Lemma \ref{p en plog} making the choices $\ex{\beta}:=\ex{p}$, $\ex{q}:=\ex{p'}$ and $\ex{p}:=1$.
	\medskip
	
	The following lemma gives a doubling property for the functional $\mathtt{f}(Q):=\normadeflp{\ca}{p}$ with $\ex{p}\inplogdern$.
	\begin{lema}[\cite{PRa}, Equation (2.11)]\label{(2.11)}
		If $\ex{p}\inplogdern$ with $p^+<\infty$, then there exists a positive constant $C_p$ such that the inequality
		\begin{equation}\label{C_p}
		\normadeflp{\car_{2Q}}{p}\leq C_p \normadeflp{\car_Q}{p}
		\end{equation}
		holds for every cube $Q\subset\z$.
	\end{lema}
	By iteration of inequality \eqref{C_p} it is not difficult to prove that
	\begin{equation}\label{C gamma}
	\normadeflp{\car_{\gamma Q}}{p}\lesssim \normadeflp{\car_Q}{p}
	\end{equation} 
	holds for every cube $Q\subset\z$, with an appropriate constant depending on $\gamma$ and $C_p$.
	\medskip
	
	The next theorem is an useful tool in order to prove Theorem \ref{teo conm variable con lipschitz generalizado}.
	\begin{teo}[\cite{DHHR}, Theorem 7.3.22]\label{7.3.22}
		If $p\inplogdern$, then 
		$$\sum_{Q\in\mathcal{D}}\normadeflp{\ca f}{p}\normadeflp{\ca g}{p'}
		\le G_p \normadeflp{f}{p}\normadeflp{g}{p'}$$
		for all $f\in\lpr{p}$, $g\in\lpr{p'}$ and every family $\mathcal{D}$ of pairwise disjoint cubes. 
	\end{teo}
	
	Moreover, a similar result considering overlaping families is the following.
	
	\begin{lema}[\cite{MP2}, Lemma 3.5]\label{lemaQ0}
		Let $\ex{p}\inplogdern$, $d\in\zZ$ and $Q_0$ a dyadic cube. If we define
		$$\mathcal{O}_d=\{Q \texto{ dyadic cube} \,:\, Q \subset Q_0 \texto{ and } \ell(Q)=2^{-d}\},$$
		then
		\begin{align}\label{32}
		\sum_{Q\in\mathcal{O}_d}
		\normadeflp{f\car_{3Q}}{p}
		\normadeflp{g\car_{3Q}}{p'}
		\lesssim
		\normadeflp{f\car_{3Q_0}}{p}
		\normadeflp{g\car_{3Q_0}}{p'}
		\end{align}
		for every $f\in\lploc{p}$ and $g\in\lploc{p'}$, where the implied constant in $\lesssim$ does not depend on $d$.
	\end{lema}
	\medskip
	
	In order to prove Lemma \ref{Izuki con lipschitz general} we state the next result that follows from [\cite{D1},Lemma 5.5].
	Recall that $f_Q$ denote the average $\frac{1}{|Q|}\int_Q f$.
	\begin{lema}[\cite{D1}]\label{lema diening con p variable}
		Let $\ex{p}\inplogdern$ with $1<p^-\le p^+<\infty$.
		Then exists a constant $0<\nu<1$ such that for every cube $Q$ and every function $f\in \lplocc{1}$ with $f_Q\neq0$,
		$$\normadeflp{|f|^\nu\car_Q}{p}
		\lesssim\normadeflp{\car_Q}{p} |f_Q|^\nu.$$
	\end{lema}
	\medskip
	
	The next theorems gives boundedness results in Musielak-Orlicz spaces for certain maximal functions.
	\begin{teo}[\cite{MMO}, Proposition 2.5]\label{teo max de HL en llogl}
		Let  $\ex{p}\inplogdern$ with $1<p^-\le p^+<\infty$ and $\ex{q}\in \mathcal{P}^{\rm loglog}(\zR^n)$.
		Then
		\begin{equation*}
		M:\lp{p}(\log L)^{\ex{q}}(\z) \hookrightarrow \lp{p}(\log L)^{\ex{q}}(\z).
		\end{equation*}
	\end{teo}
	
	\begin{teo}[\cite{DHHR}, Theorem 7.3.27]\label{teo logl}
		Let $\ex{p}, \ex{s},l(\cdot) \inplogdern$ such that $\ex{p}=\ex{s}l(\cdot)$ and $l^->1$.
		Then
		\begin{equation*}
		M_{L^{\ex{s}}}:\lp{p}(\z) \hookrightarrow \lp{p}(\z).
		\end{equation*}
	\end{teo}
	
	\begin{teo}[\cite{MP2}, Theorem 1.7]\label{teo max llogl}
		Let  $\ex{p},\ex{q}\inplogdern$ such that $\ex{p}\leq\ex{q}$ and $\ex{r}\in \mathcal{P}^{\rm loglog}(\z)$.
		Let $\ex{s}\inplogdern$ and $\beta(\cdot)$ be two functions such that $1/\beta(\cdot)=1/\ex{p}-1/\ex{q}$ and $1\leq s^-\leq s^+<p^-$.
		Then
		\begin{equation*}
		\mbetas{\beta}{L^{\ex{s}}}:\lp{p}(\log L)^{\ex{r}}(\z) \hookrightarrow \lp{q}(\log L)^{\ex{r}}(\z).
		\end{equation*}
	\end{teo}
	
	\begin{remark}\label{teo max llogl remark}
		Since $1/\beta(\cdot)=1/\ex{q'}-1/\ex{p'}$, if $1\leq s^-\leq s^+<(q')^-$ we have that
		\begin{equation*}
		\mbetas{\beta}{L^{\ex{s}}}:\lp{q'}(\log L)^{\ex{r}}(\z) \hookrightarrow \lp{p'}(\log L)^{\ex{r}}(\z).
		\end{equation*}
	\end{remark}
	\medskip
	
	The following result establishes that the spaces $\mathcal{L}^\varrho_a$ coincide, for $1\leq \varrho < \infty$.
	\begin{teo}[\cite{Li3}, Corollary 2]\label{equivalencia lipschitz con p}
		Let $1\leq \varrho < \infty$ and $a\in T_\infty$, then $\mathcal{L}^\varrho_a = \mathcal{L}^1_a$ and
		$$\sup_Q \frac{1}{a(Q)}\left(\fint_Q |b-b_Q|^\varrho\,dx \right)^{1/\varrho}\simeq
		\sup_Q \frac{1}{a(Q)}\fint_Q |b-b_Q|\,dx.$$
	\end{teo}
	\medskip
	
	The following lemma can be deduced from the proof of Theorem 2.3 in \cite{PRa} (see [\cite{PRa}, Equation (5.4)]), and it will be useful in the proof of Theorem \ref{teo general conmut lipchitz constante}.
	\begin{lema}[\cite{PRa}]\label{condicion puntual lip variable }
		Let $\ex{r}\inplogdern)$ with $r^+<\infty$ such that $r_\infty\leq \ex{r}$,
		\begin{equation*}
		1<\gamma\leq r^-\leq r^+< \frac{n\gamma}{(n-\gamma)^+} \,\,\,\texto{and}\,\,\, \ex{\delta}/n:=1/\gamma-1/\ex{r}.
		\end{equation*}
		Let $b\in\zL(\ex{\delta})$ then
		\begin{equation*}
		|b(x)-b(z)|\lesssim|x-z|^{\delta(x)}
		\end{equation*}
		for every $x,z\in\z$.
	\end{lema}
\section{Key auxiliary results}\label{S3}
In this section we give some technical lemmas that will be useful in the proof of the main results. 
\subsection{Estimates of $\normadefp{\ca}{L^{\ex{p}}(\log L)^{\ex{q}}}$}

In \cite{DHHR} the authors proved that, if $\ex{p}\inplogdern$, then   $\normadeflp{\ca}{p}\simeq|Q|^{(1/p)_Q}$ for any cube $Q$ (see [\cite{DHHR}, Lemma 4.5.3]).
Recall that $(1/p)_Q$ denotes the average $|Q|^{-1}\int_Q 1/p(x)\,dx$. 
We would like to generalize this result to the case of $L^{\ex{p}}(\log L)^{\ex{q}}$ norms, that is, estimates of $\normadefp{\ca}{L^{\ex{p}}(\log L)^{\ex{q}}}$ with $\ex{p},\ex{q}$ in certain classes of exponents.
Concretely, we prove the following result.
\medskip

\begin{propo}\label{NORMAS loglog}
	Let $\ex{p}\inplogdern$ such that $1<p^-\le p^+<\infty$ and $\ex{q}\in \mathcal{P}^{\rm loglog}(\z)$ a non-negative function.
	Then
	\begin{equation*}
	\normadefp{\ca}{L^{\ex{p}}(\log L)^{\ex{q}}}\simeq |Q|^{(1/p)_Q}(\log(e+1/|Q|))^{(q/p)_Q}.
	\end{equation*}
	for every cube $Q$ in $\z$.
\end{propo}
\medskip
\begin{remark}\label{remark NORMAS loglog}
	In particular, when $\ex{p}=\ex{q}$ with $1<p^-\le p^+<\infty$, 
	\begin{equation}\label{NORMAS loglog p igual q}
	\normadefp{\ca}{L^{\ex{p}}(\log L)^{\ex{p}}}\simeq |Q|^{(1/p)_Q}\log(e+1/|Q|)
	\end{equation}
	and if, in addition, $\ex{q}\equiv0$, 
	\begin{equation}\label{NORMAS loglog q igual 0}
	\normadefp{\ca}{L^{\ex{p}}}\simeq |Q|^{(1/p)_Q}.
	\end{equation}
	Since $\psi(t)=t\log(e+t)$ is an invertible Young function, is easy to see that 
	\begin{equation}\label{norma logl con uno}
	\normadefp{\ca}{L\log L}\simeq|Q|\log(e+1/|Q|).
	\end{equation}
\end{remark}

In order to achieve Proposition \ref{NORMAS loglog} we need the following lemmas.    

\begin{lema}\label{prop loglog}
	Let $\ex{q}\in\mathcal{P}^{\rm loglog}(\z)$ and let $Q$ be a cube in $\z$. Then, for every $x,y\in Q$,
	\begin{equation*}
	(\log(e+1/|Q|))^{q(x)}\simeq (\log(e+1/|Q|))^{q(y)}.
	\end{equation*}
\end{lema}
\medskip

\begin{proof}[Proof of Lemma \ref{prop loglog}]
	It is enough to show that there exists a positive constant $C$ such that
	\begin{equation*}
	(\log(e+1/|Q|))^{|q(x)-q(y)|}\leq C,
	\end{equation*}
	or equivalenty
	\begin{equation}\label{prop loglog 1}
	\exp[|q(x)-q(y)|\log(\log (e+1/|Q|))]\leq C.
	\end{equation}
	Since $\ex{q}\in\mathcal{P}^{\rm loglog}(\z)$,
	\begin{align}\label{prop loglog 2}
	\exp(|q(x)-q(y)|\log(\log (e+1/|Q|)))
	\leq\exp\left(C\frac{\log(e+\log (e+1/|Q|))}{\log(e+\log (e+1/|x-y|))}\right).
	\end{align}
	Since $x,y\in Q$, there exists a constant $C_n>1$ such that $|x-y|\leq C_n|Q|^{1/n}$. Then
	$$\log\left(e+\log \left(e+\frac{1}{C_n|Q|^{1/n}}\right)\right)\leq \log\left(e+\log \left(e+\frac{1}{|x-y|}\right)\right).$$
	If we prove that
	\begin{equation}\label{prop loglog 3}
	\log\left(e+\log\left(e+\frac{1}{|Q|}\right)\right)\leq \kappa \log\left(e+\log \left(e+\frac{1}{C_n|Q|^{1/n}}\right)\right)
	\end{equation}
	for some positive constant $\kappa$ then, by (\ref{prop loglog 2}), we conclude (\ref{prop loglog 1}).
	
	Let us prove inequality (\ref{prop loglog 3}). 
	Note that, since $C_n\geq1$,
	\begin{align*}
	\log\left(e+\frac{1}{|Q|}\right)
	&\lesssim \log\left(e+\frac{1}{|Q|^{1/n}}\right)
	\le \log\left(C_n\,e+\frac{C_n}{C_n|Q|^{1/n}}\right)\\
	&\le \log\left(C_n\right)\log\left(e+\frac{1}{C_n|Q|^{1/n}}\right)+\log\left(e+\frac{1}{C_n|Q|^{1/n}}\right)\\
	&\le (1+\log C_n)\log\left(e+\frac{1}{C_n|Q|^{1/n}}\right)\\
	&:= \kappa_1 \log\left(e+\frac{1}{C_n|Q|^{1/n}}\right).
	\end{align*}
	Thus, by similar argument, since $\kappa_1\ge1$, 
	\begin{align*}
	\log\left(e+\log\left(e+\frac{1}{|Q|}\right)\right)
	&\leq \log\left(e+\kappa_1\log\left(e+\frac{1}{C_n|Q|^{1/n}}\right)\right)\\
	&\leq (1+\log\kappa_1) \log\left(e+\log\left(e+\frac{1}{C_n|Q|^{1/n}}\right)\right)\\
	&:= \kappa \log\left(e+\log\left(e+\frac{1}{C_n|Q|^{1/n}}\right)\right).
	\end{align*}
\end{proof}

Let $\ex{\alpha}$ and $\ex{\theta}$ be two functions with $0<\alpha^-\le \alpha^+<\infty$ and $0\le \theta^-\le \theta^+<\infty$  and $x\in\z$, we denote
$$\phi_{\alpha(x),\theta(x)}(t):=t^{\alpha(x)}(\log(e+t))^{\theta(x)}.$$
Note that, for every fixed $x\in\z$, $\phi^{-1}_{\alpha(x),\theta(x)}(\cdot)$ is a Young function, then it is not difficult to prove that 
\begin{equation}\label{inversa fi p q}
\phi^{-1}_{\alpha(x),\theta(x)}(t)\simeq t^{1/\alpha(x)}(\log(e+t))^{-\theta(x)/\alpha(x)}
\end{equation}
(see, for example, \cite{raoren}).
If, in addition, $\alpha^->1$, 
\begin{equation}\label{conjugada fi p q}
\phi^{*}_{\alpha(x),\theta(x)}(t)\simeq t^{\alpha'(x)}(\log(e+t))^{-\theta(x)/(\alpha(x)-1)}.
\end{equation}
The constants involved in equations \eqref{inversa fi p q} and \eqref{conjugada fi p q} only depend on the extremes of the exponents $\ex{\alpha}$ and $\ex{\theta}$.
\medskip    

\begin{lema}\label{lema 4.5.1}
	Let $\ex{p}\inplogdern$ such that $1\le \ex{p}\le p^+<\infty$ and $\ex{q}\in \mathcal{P}^{\rm loglog}(\z)$ a non-negative function.
	Then for every cube $Q\subset\z$ we have
	\begin{equation*}
	\phi^{-1}_{\frac{1}{(1/p)_Q},\frac{(q/p)_Q}{(1/p)_Q}}(1/|Q|)\lesssim
	\fint_Q \phi^{-1}_{p(x),q(x)}(1/|Q|)\,dx.
	\end{equation*}
\end{lema}
\medskip

\begin{proof}
	Let $Q\subset\z$ a cube.
	Since 
	$$0< \frac{1}{(1/p)_Q}\le p^+<\infty \esp\text{and}\esp
	0\le \frac{(q/p)_Q}{(1/p)_Q}\le q^+p^+<\infty,$$ 
	by equation \eqref{inversa fi p q} with $\ex{\alpha}:=1/(1/p)_Q$ and $\ex{\theta}:=(q/p)_Q/(1/p)_Q$, we have 
	\begin{equation}\label{lema 4.5.1 1}
	\phi^{-1}_{\frac{1}{(1/p)_Q},\frac{(q/p)_Q}{(1/p)_Q}}(1/|Q|)
	\simeq
	(1/|Q|)^{(1/p)_Q}\left(\log(e+(1/|Q|))\right)^{(q/p)_Q}.
	\end{equation}
	Given $x\in Q$, define the mappings
	$$h(z):=(1/|Q|)^z\left(\log(e+(1/|Q|))\right)^{-(q/p)_Q}$$
	and 
	$$g_x(z):=(1/|Q|)^{1/p(x)}\left(\log(e+(1/|Q|))\right)^{-z}$$
	for $z\ge0$. 
	Note that, as functions of $z$, the mappings $h$ and $g_x$ are convex. 
	Thus, by \eqref{lema 4.5.1 1} and applying Jensen's inequality twice we have that
	\begin{align*}
	\phi^{-1}_{\frac{1}{(1/p)_Q},\frac{(q/p)_Q}{(1/p)_Q}}(1/|Q|)
	&\simeq h\left(\left(\frac{1}{p}\right)_Q\right)
	\leq\fint_Qh\left(\frac{1}{p(x)}\right)\,dx\\
	&=\fint_Q(1/|Q|)^{1/p(x)}\left(\log(e+(1/|Q|))\right)^{-(q/p)_Q}\,dx\\
	&=\fint_Qg_x\left(\left(\frac{q}{p}\right)_Q\right)\,dx
	\leq\fint_Q\fint_Qg_x\left(\frac{q(y)}{p(y)}\right)\,dy\,dx\\
	&=\fint_Q\fint_Q\frac{(1/|Q|)^{1/p(x)}}{(\log(e+1/|Q|))^{q(y)/p(y)}}\,dy\,dx.
	\end{align*}
	From Remark (\ref{plog y ploglog}) we can apply Lemma \ref{prop loglog} with $\ex{q}:=\ex{(q/p)}$ to obtain that
	\begin{align*}
	\phi^{-1}_{\frac{1}{(1/p)_Q},\frac{(q/p)_Q}{(1/p)_Q}}(1/|Q|)
	&\lesssim\fint_Q\frac{(1/|Q|)^{1/p(x)}}{(\log(e+1/|Q|))^{q(x)/p(x)}}\,dx
	\simeq\fint_Q \phi^{-1}_{p(x),q(x)}(1/|Q|)\,dx,
	\end{align*}
	where we have used equation \eqref{inversa fi p q} with $\ex{\alpha}:=\ex{p}$ and $\ex{\theta}:=\ex{q}$.
\end{proof}
\medskip
\begin{lema}\label{lema 4.5.2}
	Let $\ex{p},\ex{q}$ such that $1<p^-\le p^+<\infty$ and $0\le q^-\le q^+<\infty$ and let $Q$ be a cube in $\z$. Then for every $t\geq0$,
	\begin{equation*}
	t\lesssim
	\fint_Q \phi^{-1}_{p(x),q(x)}(t)\,dx\,\,
	\fint_Q (\log(e+t))^{q(x)}\phi^{-1}_{p'(x),q(x)}(t)\,dx.
	\end{equation*}
\end{lema}
\begin{proof}
	It is enough to prove the case $t > 0$.
	Since, by equation \eqref{inversa fi p q}, 
	\begin{align*}
	\phi^{-1}_{p(x),q(x)}(t)\,\,\phi^{-1}_{p'(x),q(x)}(t)
	&\simeq\frac{t^{1/p(x)}}{(\log(e+t))^{q(x)/p(x)}}
	\frac{t^{1/p'(x)}}{(\log(e+t))^{q(x)/p'(x)}}\\
	&=\frac{t}{(\log(e+t))^{q(x)}},
	\end{align*}
	then, by Jensen's inequality, we have
	\begin{align*}
	\fint_Q \phi^{-1}_{p(x),q(x)}(t)\,dx
	&\simeq t\fint_Q \frac{1}{(\log(e+t))^{q(x)}\phi^{-1}_{p'(x),q(x)}(t)}\,dx\\
	&\gtrsim t\frac{1}{\fint_Q(\log(e+t))^{q(x)}\phi^{-1}_{p'(x),q(x)}(t)\,dx}.
	\end{align*}
\end{proof}

\begin{proof}[Proof of Proposition~\ref{NORMAS loglog}]
	Let $Q$ be a cube in $\z$, define
	$$f(x):=\ca(x)\phi^{-1}_{p(x),q(x)}(1/|Q|),\esp x\in\z$$
	and
	$$g(x):=\ca(x)(\log(e+1/|Q|))^{q(x)}\phi^{-1}_{p'(x),q(x)}(1/|Q|),\esp x\in\z.$$
	Note that $\normadefp{f}{L^{\ex{p}}(\log L)^{\ex{q}}}\leq 1$ and $\normadefp{g}{L^{\ex{p'}}(\log L)^{-\ex{q}/(\ex{p}-1)}}\leq C_2$ with $C_2$ a positive constant independent of $Q$. Indeed, since by \eqref{inversa fi p q}, 
	\begin{align*}
	\int_{\z} \phi_{p(x),q(x)}(f(x))\,dx
	&=\int_{Q}\phi_{p(x),q(x)}\left(\phi^{-1}_{p(x),q(x)}(1/|Q|)\right)\,dx
	\simeq 1,
	\end{align*}
	the estimation for $f$ is clear. 
	Note that, for $x\in Q$, by \eqref{inversa fi p q},
	\begin{align*}
	\log(e+g(x))
	&=\log\left[e+(\log(e+1/|Q|))^{q(x)}\phi^{-1}_{p'(x),q(x)}(1/|Q|)\right]\\
	&\simeq\log\left[e+(\log(e+1/|Q|))^{q(x)/p(x)}(1/|Q|)^{1/p'(x)}\right]\\
	&\ge\log\left(e+(1/|Q|)^{1/p'(x)}\right)
	\ge\frac{1}{(p')^+}\log\left(e+1/|Q|\right)\\
	&\gtrsim\log(e+1/|Q|),
	\end{align*}
	since $(p')^+<\infty$.
	Thus we have that
	\begin{align*}
	\int_Q\frac{g^{p'(x)}}{(\log(e+g))^{q(x)/(p(x)-1)}}\,dx
	&\lesssim\int_Q\frac{g^{p'(x)}}{(\log(e+1/|Q|))^{q(x)/(p(x)-1)}}\,dx\\
	&\lesssim\fint_Q(\log(e+1/|Q|))^{q(x)\left(p'(x)-\frac{1}{p(x)-1}-1\right)}\,dx
	\lesssim 1,
	\end{align*}
	since $p'(x)-{1}/({p(x)-1})-1=0$.
	
	By Lemma \ref{lema 4.5.2} with $t:=1/|Q|$ we have
	\begin{align}\label{normasloglog12}
	1
	&\lesssim|Q|
	\fint_Q \phi^{-1}_{p(x),q(x)}(1/|Q|)\,dx
	\,\fint_Q (\log(e+1/|Q|))^{q(x)}\,\phi^{-1}_{p'(x),q(x)}(1/|Q|)\,dx\nonumber\\
	&=
	\fint_Q f\,dx
	\,\int_{\z} \ca(x)g(x)\,dx.
	\end{align}
	We can apply H\"{o}lder's inequality  \eqref{holder fi y fi conjugada} with $\Psi(x,t):=t^{p(x)}(\log(e+t))^{q(x)}$ and $\Psi^*(x,t):=t^{p'(x)}(\log(e+t))^{-q(x)/(p(x)-1)}$ (see equation \eqref{conjugada fi p q}), to obtain             	
	\begin{align}\label{normasloglog1}
	1&\lesssim
	\fint_Q f(x)\,dx
	\normadeflplogq{\ca}{p}{q}
	\normadefp{g}{L^{\ex{p'}}(\log L)^{-\ex{q}/(\ex{p}-1)}}\nonumber\\
	&\lesssim \normadeflplogq{f_Q\ca}{p}{q}
	\le \normadeflplogq{Mf}{p}{q}            
	\nonumber\\
	&\lesssim \normadeflplogq{f}{p}{q}\lesssim 1
	\end{align}
	where we have used Theorem \ref{teo max de HL en llogl}.
	
	Since $\fint_Q f(x)\,dx=\fint_Q\phi^{-1}_{p(x),q(x)}(1/|Q|)\,dx> 0$, from equation \eqref{normasloglog1} we obtain that
	\begin{align}\label{normas loglog 1}
	&|Q|\fint_Q(\log(e+1/|Q|))^{q(x)}\,\phi^{-1}_{p'(x),q(x)}(1/|Q|)\,dx\nonumber\\               &\quad\quad  \lesssim\normadefp{\ca}{L^{\ex{p}}(\log L)^{\ex{q}}}
	\lesssim \frac{1}{\fint_Q\phi^{-1}_{p(x),q(x)}(1/|Q|)\,dx}.
	\end{align}
	By Lemma \ref{lema 4.5.1} we can estimate the right-hand side of inequality \eqref{normas loglog 1} using equation \eqref{inversa fi p q} as follow
	\begin{align*}
	\frac{1}{\fint_Q\phi^{-1}_{p(x),q(x)}(1/|Q|)\,dx}
	&\lesssim\frac{1}{\phi^{-1}_{\frac{1}{(1/p)_Q},\frac{(q/p)_Q}{(1/p)_Q}}(1/|Q|)}\\
	&\simeq|Q|^{(1/p)_Q}(\log(e+1/|Q|))^{(q/p)_Q}.
	\end{align*}
	In order to estimate the left-hand side of inequality \eqref{normas loglog 1}, if $x\in Q$, by Jensen's inequality and Lemma \ref{prop loglog},
	\begin{align*}
	(\log(e+1/|Q|))^{q_Q}
	&\leq\fint_Q(\log(e+1/|Q|))^{q(y)}\,dy\\
	&\simeq (\log(e+1/|Q|))^{q(x)}.
	\end{align*}
	Thus by Lemma \ref{lema 4.5.1} we have
	\begin{align*}
	&|Q|\fint_Q(\log(e+1/|Q|))^{q(x)}\,\phi^{-1}_{p'(x),q(x)}(1/|Q|)\,dx\\
	&\quad\quad\gtrsim|Q|(\log(e+1/|Q|))^{q_Q}\,\fint_Q\phi^{-1}_{p'(x),q(x)}(1/|Q|)\,dx\\
	&\quad\quad\gtrsim|Q|(\log(e+1/|Q|))^{q_Q}\,\,\phi^{-1}_{\frac{1}{(1/p')_Q},\frac{(q/p')_Q}{(1/p')_Q}}(1/|Q|)\\
	&\quad\quad\gtrsim|Q|(\log(e+1/|Q|))^{q_Q}\,|Q|^{-(1/p')_Q}\,(\log(e+1/|Q|))^{-(q/p')_Q}\\
	&\quad\quad\simeq |Q|^{(1/p)_Q}\,(\log(e+1/|Q|))^{(q/p)_Q}
	\end{align*}
\end{proof}

\begin{cor}\label{corollary 3.5}
	Let $\ex{p}\inplogdern$ with $p^+<\infty$ and let $Q$ be a cube in $\z$. Then
	\begin{equation*}
	\fint_Q|Q|^{1/p(x)}\,dx\lesssim\normadefp{\ca}{\ex{p}}.
	\end{equation*}
\end{cor}
\begin{proof}
	From the proof of Proposition \ref{NORMAS loglog}, by using inequality (\ref{normasloglog1}) with $\ex{p}:=\ex{p'}$ and $\ex{q}:=0$ we have
	\begin{align}\label{cor1}
	\normadefp{\ca}{\ex{p'}}
	\fint_Q(1/|Q|)^{1/p'(x)}\,dx\lesssim 1.
	\end{align}
	Since
	\begin{align*}
	\int_Q(1/|Q|)^{1/p'(x)}\,dx&=\int_Q|Q|^{1/p(x)-1}\,dx=\fint_Q|Q|^{1/p(x)}\,dx
	\end{align*}
	by (\ref{cor1}) we obtain that
	\begin{align*}
	\frac{\normadefp{\ca}{\ex{p'}}}{|Q|}
	\fint_Q|Q|^{1/p(x)}\,dx\lesssim 1.
	\end{align*}
	Thus, by Lemma \ref{p en plog},
	\begin{align*}
	\fint_Q|Q|^{1/p(x)}\,dx\lesssim \normadefp{\ca}{\ex{p}}.
	\end{align*}		
\end{proof}

\bigskip
We now show that the Examples \textit{\ref{ejemplo 1}} and \textit{\ref{ejemplo 2}} satisfy the hypotheses of Theorem \ref{teo general conmut lipchitz constante}.

Let us see \textit{\ref{ejemplo 1}}.
Recall that, for $\ex{p}\inplogdern$ with $1<p^-\le p^+<\infty$ and $\sigma>(p')^+/(p')^-$, 
$A_1(x,t)=t^{\sigma p'(x)}(\log(e+t))^{\sigma p'(x)}$,
$B_1(x,t)=t^{(\sigma p')'(x)}$ and
$D_1(t)=t \log(e+t)$.

If we define 
$\ex{s}:=\ex{(\sigma p')'}$ and 
$\ex{l}:=\ex{p}/\ex{s}$, by Lemma \ref{propiedades plog}(ii) and (iv), $\ex{s},\ex{l}\inplogdern$.
Moreover, $l^->1$.
In fact, since $\sigma>(p')^+/(p')^-$, 
$$(p^-)'=(p')^+<\sigma(p')^-=(\sigma p')^-$$
which implies that
$$p^->[(\sigma p')^-]'=[(\sigma p')']^+$$
and then 
$$1<\frac{p^-}{[(\sigma p')']^+}\le l^-.$$
Thus, we can apply	Theorem \ref{teo logl} and Theorem \ref{teo max llogl} to obtain that
$$M_{L^{\ex{(\sigma p')'}}}:\lpr{p}\rightarrow\lpr{p}.$$
and
$$M_{\ex{\beta},L^{\ex{(\sigma p')'}}}:\lpr{p}\rightarrow\lpr{p},$$
respectively.	
Condition \textit{\ref{D y D conjugada}} it follows from Lemma \ref{HH 4.4.5 }.
By Remark \ref{remark NORMAS loglog},
\begin{align*}
\normadeflpl{\ca}{A_1}\normadeflpl{\ca}{B_1}
&=\normadefp{\ca}{L^{\ex{\sigma p'}}(\log L)^{\ex{\sigma p'}}}
\normadefp{\ca}{L^{\ex{(\sigma p')'}}}\\
&\simeq
|Q|^{(1/\sigma p')_Q}\log(e+1/|Q|)
|Q|^{(1/(\sigma p')')_Q}\\
&=
|Q|\log(e+1/|Q|)\\
&\simeq
\normadefp{\ca}{L\log L}
=\normadefp{\ca}{D_1(L)},
\end{align*}
by equation \eqref{norma logl con uno},
and thus condition \textit{\ref{desigualdad normas}} is satisfied.
Condition \textit{\ref{a b y c para holder}} follows from the fact that, by equation \ref{inversa fi p q},
$$A_1^{-1}(x,t)B_1^{-1}(x,t)
\simeq\frac{t^{1/\sigma p'(x)}}{\log (e+t)}
t^{1/(\sigma p')'(x)}				
=\frac{t}{\log(e+t)}\simeq D_1^{-1}(t).$$

\medskip

Let us now see \textit{\ref{ejemplo 2}}.
Recall that for $\ex{p}\inplogdern$ with $1<p^-\le p^+<\infty$  and $\sigma>(p')^+/(p')^-$, $\ex{\mu}\inplogdern$ such that $1<\mu^-\le\mu^+<\infty$ and 
\begin{equation}\label{alfa inf mayor a uno ov}
1/\sigma\ex{p'}-1/\ex{\mu}>\epsilon,
\end{equation}
for some constant $\epsilon\in(0,1)$ and
$\ex{\nu}\inplogdern$,
$A_2(x,t)=t^{\mu(x)}(\log(e+t))^{\nu(x)\mu(x)}$, $B_2(x,t)=t^{(\sigma p')'(x)}$ and $D_2(x,t)=t^{\alpha(x)}(\log(e+t))^{\alpha(x)\nu(x)}$ 
where $\ex{\alpha}$ is defined by
$1/\ex{\alpha}=1/\ex{\mu}+1/\ex{(\sigma p')'}.$

Note that, by Lemma \ref{propiedades plog}(iii), $\ex{\alpha}\inplogdern$. 
Moreover, $1<\alpha^-\le\alpha^+<\infty$.
In fact, by inequality \eqref{alfa inf mayor a uno ov},
\begin{align*}
\frac{1}{\ex{\alpha}}=
\frac{1}{\ex{\mu}}+
\frac{1}{\ex{(\sigma p')'}}
<\frac{1}{\ex{\sigma p'}}
+\frac{1}{\ex{(\sigma p')'}}-\varepsilon
=1-\varepsilon.
\end{align*}
Thus, $\alpha^-\ge 1/(1-\varepsilon)>1$.
Also, 
\begin{align*}
{\ex{\alpha}}=
\frac{\ex{\mu}\ex{(\sigma p')'}}{\ex{\mu}+\ex{(\sigma p')'}}
\le 
{{\mu}^+}<\infty.
\end{align*}
Then, by Remark \ref{plog y ploglog},
$\ex{(\alpha\nu)},\ex{(\mu\nu)}\in\mathcal{P}^{\rm loglog}(\z)$.
Thus, by Proposition \ref{NORMAS loglog} and equation \eqref{NORMAS loglog q igual 0}, we have
\begin{align*}
\normadeflpl{\ca}{A_2}\normadeflpl{\ca}{B_2}
&=\normadefp{\ca}{L^{\mu(\cdot)}(\log L)^{\ex{(\mu\nu)}}}
\normadefp{\ca}{L^{\ex{(\sigma p')'}}}\\
&\simeq
|Q|^{(1/\mu)_Q}(\log(e+1/|Q|))^{\nu_Q}
|Q|^{(1/(\sigma p')')_Q}\\
&\simeq
|Q|^{(1/\alpha)_Q}(\log(e+1/|Q|))^{\nu_Q}\\
&\simeq
\normadefp{\ca}{L^{\ex{\alpha}}(\log L)^{\ex{(\alpha\nu)}}}
\simeq\normadeflpl{\ca}{D_2}.
\end{align*}
Then \textit{\ref{desigualdad normas}} holds.
On the other hand, by equation \eqref{inversa fi p q},
$$A_2^{-1}(x,t)B_2^{-1}(x,t)
\simeq\frac{t^{1/\mu(x)}}{(\log (e+t))^{\nu(x)}}t^{1/(\sigma p')'(x)}
\simeq\frac{t^{1/\alpha(x)}}{(\log(e+t))^{\nu(x)}}
\simeq D_2^{-1}(x,t),$$
thus \textit{\ref{a b y c para holder}} holds.
Note that, by Lemma \ref{teo max de HL en llogl} with $\ex{p}:=\ex{\alpha}$ and $\ex{q}:=\ex{(\alpha\nu)}$, $M:L^{\ex{\alpha}}(\log L)^{\ex{(\alpha\nu)}}(\z)\rightarrow L^{\ex{\alpha}}(\log L)^{\ex{(\alpha\nu)}}(\z)$.
Thus, by duality (see equation \eqref{dual}), we have that 
\begin{align*}
\normadeflpl{\ca}{D_2}\normadeflpl{\ca}{D_2^*}
&\lesssim
\normadeflpl{\ca}{D_2}
\sup_{\normadeflpl{g}{D_2}\leq1}\int_Q |g(x)|\,dx\\
&= \sup_{\normadeflpl{g}{D_2}\leq1}
\left\Vert\ca \int_Q |g(x)|\,dx\right\Vert_{D_2(\cdot,L)}\\
&= |Q|
\sup_{\normadeflpl{g}{D_2}\leq1}
\left\Vert\ca
\frac{1}{|Q|}
\int_Q |g(x)|\,dx\right\Vert_{D_2(\cdot,L)}\\
&\le |Q|\sup_{\normadeflpl{g}{D_2}\leq1} \normadeflpl{\ca Mg}{D_2}
\le |Q|.
\end{align*}
Then condition \textit{\ref{D y D conjugada}} holds.		
\subsection{Estimates in $\zL(\ex{\delta})$}
We now give some previous estimates for the symbol functions we are interested in.
\begin{lema}\label{Izuki con lipschitz general}
	Let $k$ be a positive integer and $\ex{p}\inplogdern$ with $1<p^-\le p^+<\infty$.
	Let $a\in T_\infty$ and $b\in \mathcal{L}^1_a$.
	Then, for every cube $Q\subset\z$,
	\begin{equation}
	\frac{\normadeflp{\car_Q(b-b_Q)^k}{p}}{\normadeflp{\car_Q}{p}}
	\lesssim \left(a(Q)\normadefp{b}{\mathcal{L}^1_a}\right)^k.
	\end{equation}
\end{lema}

\begin{proof}
	Let $Q$ be a fixed cube.
	By Lemma \ref{lema diening con p variable} there exist a constant $0<\nu<1$ independent of $Q$ such that for all $f\in\lplocc{1}$,
	\begin{equation}\label{Izuki con lipschitz2}
	\normadeflp{\car_{Q}|f|^\nu}{p}
	\lesssim \left(|f|_{Q}\right)^\nu\normadeflp{\car_{Q}}{p}.
	\end{equation}
	
	We now put $f(x)=(b(x)-b_Q)^{k/\nu}$.
	Noticing that $k/\nu>1$, by Theorem \ref{equivalencia lipschitz con p}, we have
	\begin{align*}
	\left(|f|_{Q}\right)^\nu
	&=\left(\frac{1}{|Q|}\int_Q|b(x)-b_Q|^{k/\nu}\,dx\right)^\nu\\
	&=\left[\frac{a(Q)}{a(Q)}
	\left(\frac{1}{|Q|}\int_Q|b(x)-b_Q|^{k/\nu}\,dx\right)^{\nu/k}\right]^k\\
	&\simeq\left[a(Q)
	\left(\frac{1}{a(Q)|Q|}\int_Q|b(x)-b_Q|\,dx\right)\right]^k               \lesssim\left[a(Q)\normadefp{b}{\mathcal{L}^1_a}\right]^k.
	\end{align*}
\end{proof}

\begin{lema}\label{Lemma 3.7 PRa con a general}
	Let $a\in T_\infty$ and $b\in \mathcal{L}^1_a$, then the following inequality
	$$\,|b_{3Q}-b_Q|\lesssim\|a\|_{\mathbf{t}_\infty}\,a(3Q)\normadefp{b}{\mathcal{L}^1_a}.$$
	holds for every cube $Q\subset\z$.
\end{lema}
\begin{proof}
	Let $Q$ be a fixed cube.
	Then, by $T_\infty$ condition \eqref{cond t infinito}, we have that
	\begin{align*}
	|b_{3Q}-b_Q|
	&\leq
	|b_{3Q}-b_{2Q}|+|b_{2Q}-b_Q|\\
	&\leq \frac{1}{|2Q|}\int_{2Q}|b(x)-b_{3Q}|\,dx+
	\frac{1}{|Q|}\int_{Q}|b(x)-b_{2Q}|\,dx\\
	&\lesssim \frac{1}{|3Q|}\int_{3Q}|b(x)-b_{3Q}|\,dx+
	\frac{1}{|2Q|}\int_{2Q}|b(x)-b_{2Q}|\,dx\\
	&\lesssim a(3Q)\normadefp{b}{\mathcal{L}^1_a}+
	a(2Q)\normadefp{b}{\mathcal{L}^1_a}\\
	&\lesssim\|a\|_{t_\infty}\,a(3Q)\normadefp{b}{\mathcal{L}^1_a}.
	\end{align*}
\end{proof}

In the proof of Theorem \ref{teo general conmut lipchitz constante} we shall use the following pointwise estimate for $b\in\mathbb{L}(\ex{\delta})$.

\begin{lema}\label{obs teorema delta variable}
	Let $\ex{r}\inplogdern$ with $r_\infty\leq\ex{r}\le r^+<\infty$ and $\ex{\delta}$ be defined as in (\ref{delta}) and $b\in\mathbb{L}(\ex{\delta})$. Let $Q$ be a cube in $\z$ and $z\in kQ$ for some positive integer $k$.
	Then
	\begin{align*}
	\left|b(z)-b_Q\right|
	&\lesssim\normadefp{\ca}{n/\ex{\delta}}.
	\end{align*}		
\end{lema}
\begin{proof}
	Note that if $x\in Q$ and $z\in kQ$ for some positive integer $k$, then $|z-x|\lesssim|Q|^{1/n}$.
	Thus by Lemma \ref{condicion puntual lip variable } and Corollary \ref{corollary 3.5} we have
	\begin{align*}
	\left|b(z)-b_Q\right|
	&\leq\fint_Q\left|b(z)-b(x)\right|\,dx\nonumber\lesssim\fint_Q\left|z-x\right|^{\delta(x)}\,dx\nonumber\\
	&\lesssim\fint_Q|Q|^{\delta(x)/n}\,dx
	\lesssim\normadefp{\ca}{n/\delta(\cdot)}.
	\end{align*}
\end{proof}
\section{Proof of main results}\label{S4}
In this section we present the proofs of Theorem \ref{teo conm variable con lipschitz generalizado} and Theorem \ref{teo general conmut lipchitz constante}.
\begin{proof}[Proof of Theorem~\ref{teo conm variable con lipschitz generalizado}]
	Since $v\in \lploc{p}$ implies that the set of bounded functions with compact support is dense in $\lpw{p}{v}(\z)$, it is enough to show that 
	$$\normadeflpw{\tfi^{b,m} f}{q}{w}\lesssim \normadeflpw{f}{p}{v}$$
	for each non-negative bounded function with compact support $f$.
	Moreover, by duality (see equation \eqref{dual}) this is equivalent to prove that
	$$\int_{\z} |\tfi^{b,m} f(x)| w(x)g(x)\,dx \lesssim \normadeflpw{f}{p}{v}$$
	for all non-negative bounded functions with compact support $f,g$ such that $\normadeflp{g}{q'}\leq1$.
	
	Let $\overline{K}$ be the function defined by
	$$\overline{K}(t)=\sup_{t<|x|\leq2t}K(x),$$
	for every $t>0$. It was proved in [\cite{Li}, Proof of Theorem 2.2] that, if $K\in \mathfrak{D}$, we can estimate the commutator as follows
	$$|\tfi^{b,m} \fx|\leq\sum_Q \overline{K}\left(\frac{\ell(Q)}{2}\right)\sum_{j=0}^m{m \choose j} |b(x)-b_Q|^{m-j}\car_Q(x)\int_{3Q}|b(z)-b_Q|^{j}f(z)\,dz,$$
	where  the sum is taken over all dyadic cubes of $\z$.
	Hence
	\begin{align}\label{111}
	& \int_{\zR^n} |\tfi^{b,m}  f(x) |w(x)g(x)\,dx \nonumber\\
	&\quad\lesssim\sum_Q \overline{K}\left(\frac{\ell(Q)}{2}\right)
	\sum_{j=0}^m
	\int_{3Q}|b(z)-b_Q|^{j}f(z)\,dz\,\int_Q |b(x)-b_Q|^{m-j}g(x)w(x)\,dx.
	\end{align}
	Let us denote $s(\cdot):= Rp'(\cdot)$ and $\ex{l}:= S\ex{q}$.
	Since $(p')^+<R(p')^-$ and $q^+<Sq^-$ then $(s')^+<p^-$ and $(l')^+<(q^+)'$.
	Let $\mu,\nu$ two constants such that
	$$(s')^+< \mu <p^-
	\esp\esp\texto{ and }\esp\esp
	(l')^+<\nu<(q^+)',$$
	and $\omega(\cdot),\tau(\cdot)$ defined by
	$$\frac{1}{\omega(\cdot)}=\frac{1}{s(\cdot)}+\frac{1}{\mu}
	\esp\esp\texto{ and }\esp\esp
	\frac{1}{\tau(\cdot)}=\frac{1}{l(\cdot)}+\frac{1}{\nu}.$$
	Observe that, by Lemma \ref{propiedades plog},  $\ex{\omega},\ex{\tau}\inplogdern$ since $\ex{s},\ex{l}\inplogdern$.
	Using H\"{o}lder's inequality \eqref{holderpp} twice and Lemma \ref{p en plog}, we can estimate (\ref{111}) by a multiple of
	\begin{align}\label{112}
	& \sum_Q \overline{K}\left(\frac{\ell(Q)}{2}\right)
	\sum_{j=0}^m |3Q|
	\frac{\normadeflp{\car_{3Q} |b-b_Q|^{j}}{\omega'}}
	{\normadeflp{\car_{3Q}}{\omega'}}
	\frac{\normadeflp{\car_{3Q} f}{\omega}}
	{\normadeflp{\car_{3Q}}{\omega}}\nonumber\\
	&\quad\quad\times
	|Q|
	\frac{\normadeflp{\car_Q |b-b_Q|^{m-j}}{\tau'}}
	{\normadeflp{\car_Q}{\tau'}}
	\frac{\normadeflp{\car_Q gw}{\tau}}
	{\normadeflp{\car_Q}{\tau}}.
	\end{align}
	Notice that, by Lemmas \ref{Izuki con lipschitz general} and \ref{Lemma 3.7 PRa con a general}, we have
	\begin{align*}
	\frac{\normadeflp{\car_{3Q} |b-b_Q|^{j}}{\omega'}}
	{\normadeflp{\car_{3Q}}{\omega'}}
	&\lesssim \frac{\normadeflp{\car_{3Q} |b-b_{3Q}|^j}{\omega'}}
	{\normadeflp{\car_{3Q}}{\omega'}}
	+\frac{\normadeflp{\car_{3Q} |b_{3Q}-b_Q|^j}{\omega'}}
	{\normadeflp{\car_{3Q}}{\omega'}}\\
	&\lesssim \left(\|a\|_{\textit{t}_\infty}a(3Q)
	\normadefp{b}{\mathcal{L}^1_a}\right)^j.
	\end{align*}
	Thus, since $a\in T_\infty$, we can estimate (\ref{112}) as follows
	\begin{align}\label{113}
	& \int_{\z} |\tfi^{b,m}  f(x) |w(x)g(x)\,dx \nonumber\\
	&\quad \lesssim\sum_Q \overline{K}\left(\frac{\ell(Q)}{2}\right)
	\sum_{j=0}^m |3Q|
	\left(\|a\|_{\textit{t}_\infty}a(3Q)\normadefp{b}{\mathcal{L}^1_a}\right)^j
	\frac{\normadeflp{\car_{3Q} f}{\omega}}
	{\normadeflp{\car_{3Q}}{\omega}}\nonumber\\
	&\quad\quad\quad
	\times |Q|\left(\|a\|_{\textit{t}_\infty}a(Q)\normadefp{b}{\mathcal{L}^1_a}\right)^{m-j}
	\frac{\normadeflp{\car_Q gw}{\tau}}
	{\normadeflp{\car_Q}{\tau}}\nonumber\\
	&\quad
	\lesssim \normadefp{b}{\mathcal{L}^1_a}^m
	\sum_Q a(3Q)^m\,\, \overline{K}\left(\frac{\ell(Q)}{2}\right)|3Q|
	\frac{\normadeflp{\car_{3Q} f}{\omega}}
	{\normadeflp{\car_{3Q}}{\omega}}
	|Q|
	\frac{\normadeflp{\car_Q gw}{\tau}}
	{\normadeflp{\car_Q}{\tau}}.
	\end{align}
	Since $g$ has compact support and $w\in\lploc{Sq}$,
	$$\lim_{\ell(Q)\rightarrow \infty} \frac{\normadeflp{\car_Q gw}{\tau}}{\normadeflp{\car_Q}{\tau}}= 0.$$
	Let $C_\tau$, $C_\tau^*$, $C_\tau^{**}$ and $G_\tau$ be the constants provided by Lemma \ref{(2.11)}, Lemma \ref{p en plog} and Theorem \ref{7.3.22} respectively.
	If $\alpha> C_\tau C_\tau^* C_\tau^{**}G_\tau$ and $k\in\zZ$, it follows that, if for some dyadic cube $Q$,
	\begin{equation}\label{114}
	\alpha^k<\frac{\normadeflp{\car_Q gw}{\tau}}{\normadeflp{\car_Q}{\tau}},
	\end{equation}
	then $Q$ is contained in dyadic cubes satisfying this condition, which are maximal with respect to the inclusion.
	Thus, for each integer $k$ there is a family of maximal non-overlapping dyadic cubes $\{Q_{k,j}\}_{j\in\zZ}$ satisfying \eqref{114}.
	Let $Q_{k,j}'$ be the dyadic cube containing $Q_{k,j}$ with sidelength $2\ell(Q_{k,j})$. Then, by maximality and Lemma \ref{(2.11)}, we have
	\begin{align*}
	\alpha^k
	&<\frac{\normadeflp{\car_{Q_{k,j}} gw}{\tau}}
	{\normadeflp{\car_{Q_{k,j}}}{\tau}}
	\leq\frac{\normadeflp{\car_{Q_{k,j}'}}{\tau}}
	{\normadeflp{\car_{Q_{k,j}}}{\tau}}
	\frac{\normadeflp{\car_{Q_{k,j}'} gw}{\tau}}
	{\normadeflp{\car_{Q_{k,j}'}}{\tau}}\leq C_\tau\, \alpha^k
	\leq \alpha^{k+1}.
	\end{align*}
	For $k\in\zZ$ we define the set
	$$\mathcal{C}_k:=\left\{Q \texto{  dyadic } \,:\, \alpha^k<\frac{\normadeflp{\car_{Q} gw}{\tau}}{\normadeflp{\car_{Q}}{\tau}}\leq \alpha^{k+1} \right\}.$$
	Then every dyadic cube $Q$ for which $\normadeflp{\car_{Q} gw}{\tau}/\normadeflp{\car_{Q}}{\tau}\neq0$ belongs to exactly one $\mathcal{C}_k$. Furthermore, if $Q\in\mathcal{C}_k$, it follows that $Q\subset Q_{k,j}$ for some $j$.
	Then, from (\ref{113}) and $T_\infty$ condition \eqref{cond t infinito}, we obtain that
	\begin{align}\label{115}
	& \int_{\z} |\tfi^{b,m}  f(x) |w(x)g(x)\,dx \nonumber \\
	&\quad\lesssim\normadefp{b}{\mathcal{L}^1_a}^m
	\sum_{k\in\zZ} \sum_{Q\in \mathcal{C}_k}
	a(3Q)^m \,\overline{K}\left(\frac{\ell(Q)}{2}\right)|3Q|
	\frac{\normadeflp{\car_{3Q} f}{\omega}}
	{\normadeflp{\car_{3Q}}{\omega}}
	|Q|
	\frac{\normadeflp{\car_Q gw}{\tau}}
	{\normadeflp{\car_Q}{\tau}}
	\nonumber\\
	&\quad \lesssim\normadefp{b}{\mathcal{L}^1_a}^m
	\sum_{(k,j)\in\zZ\times\zZ}\alpha^{k+1}
	\sum_{Q\in \mathcal{C}_k\,:\, Q\subset Q_{k,j}} a(3Q)^m \overline{K}\left(\frac{\ell(Q)}{2}\right)
	|3Q||Q|
	\frac{\normadeflp{\car_{3Q} f}{\omega}}
	{\normadeflp{\car_{3Q}}{\omega}}\nonumber\\
	&\quad \lesssim
	\normadefp{b}{\mathcal{L}^1_a}^m\alpha
	\sum_{(k,j)\in\zZ\times\zZ}
	\frac{\normadeflp{\car_{Q_{k,j}} gw}{\tau}}
	{\normadeflp{\car_{Q_{k,j}}}{\tau}}
	a(3Q_{k,j})^m \nonumber\\
	&\quad\quad\quad\times
	\sum_{Q\in \mathcal{C}_k\,:\, Q\subset Q_{k,j}} \overline{K}\left(\frac{\ell(Q)}{2}\right)
	|3Q||Q|
	\frac{\normadeflp{\car_{3Q} f}{\omega}}
	{\normadeflp{\car_{3Q}}{\omega}}.
	\end{align}
	If we show that there is a constant $C_K$ such that, for any dyadic cube $Q_0$,
	\begin{align}\label{116}
	&\sum_{Q\,:\, Q\subset Q_0}
	\overline{K}\left(\frac{\ell(Q)}{2}\right)
	|3Q||Q|
	\frac{\normadeflp{\car_{3Q} f}{\omega}}{\normadeflp{\car_{3Q}}{\omega}}\nonumber\\
	&\quad\quad\leq C_K
	\widetilde{K}(\delta(1+\varepsilon)\ell(Q_0))
	|3Q_0|
	\frac{\normadeflp{\car_{3Q_0} f}{\omega}}{\normadeflp{\car_{3Q_0}}{\omega}},
	\end{align}
	with $\varepsilon,\delta$ the numbers provided by condition $\mathfrak{D}$ and $\widetilde{K}(t)=\int_{|z|\leq t} K(z)\,dz$, from (\ref{115}) we obtain that
	\begin{align}\label{estrella}
	&\int_{\z} |\tfi^{b,m}  f(x) |w(x)g(x)\,dx\nonumber\\
	&\quad \lesssim\normadefp{b}{\mathcal{L}^1_a}^m
	\sum_{(k,j)\in\zZ\times\zZ}
	a(3Q_{k,j})^m
	C_K
	\widetilde{K}(\delta(1+\varepsilon)\ell(Q_{k,j}))
	|3Q_{k,j}|\nonumber\\
	&\quad \quad \quad\times \frac{\normadeflp{\car_{3Q_{k,j}} f}{\omega}}
	{\normadeflp{\car_{3Q_{k,j}}}{\omega}}
	\frac{\normadeflp{\car_{Q_{k,j}} gw}{\tau}}
	{\normadeflp{\car_{Q_{k,j}}}{\tau}}.
	\end{align}
	Let $\gamma=\max \{3,\delta(1+\varepsilon)\}$.
	Note that $\widetilde{K}$ is an increasing function.
	From \eqref{estrella}, by Lemma \ref{(2.11)}, $T_\infty$ condition (\ref{cond t infinito}), H\"{o}lder's inequality and Lemma \ref{equivalenciabeta} we have that
	\begin{align*}
	& \int_{\zR^n} |\tfi^{b,m}  f(x) |w(x)g(x)\,dx \nonumber \\
	&\quad\lesssim\normadefp{b}{\mathcal{L}^1_a}^m
	\sum_{(k,j)\in\zZ\times\zZ}
	a(\gamma Q_{k,j})^m\,
	\widetilde{K}(\gamma \ell(Q_{k,j}))
	|\gamma Q_{k,j}|
	\frac{\normadeflp{\car_{\gamma Q_{k,j}} f}{\omega}}
	{\normadeflp{\car_{\gamma Q_{k,j}}}{\omega}}
	\frac{\normadeflp{\car_{\gamma Q_{k,j}} gw}{\tau}}
	{\normadeflp{\car_{\gamma Q_{k,j}}}{\tau}}\nonumber\\
	&\quad \lesssim\normadefp{b}{\mathcal{L}^1_a}^m
	\sum_{(k,j)\in\zZ\times\zZ}
	a(\gamma Q_{k,j})^m
	\widetilde{K}(\gamma \ell(Q_{k,j}))
	|\gamma Q_{k,j}|
	\frac{\normadefp{\car_{\gamma Q_{k,j}} fv}{\mu}}
	{\normadefp{\car_{\gamma Q_{k,j}}}{\mu}}
	\frac{\normadeflp{\car_{\gamma Q_{k,j}} v^{-1}}{s}}
	{\normadeflp{\car_{\gamma Q_{k,j}}}{s}}\\
	&\quad\quad\times
	\frac{\normadefp{\car_{\gamma Q_{k,j}} g}{\nu}}
	{\normadefp{\car_{\gamma Q_{k,j}}}{\nu}}
	\frac{\normadeflp{\car_{\gamma Q_{k,j}} w}{l}}
	{\normadeflp{\car_{\gamma Q_{k,j}}}{l}}.
	\end{align*}
	Thus, by Fefferman-Phong type condition \eqref{1.z} on the weights we obtain
	\begin{align*}
	& \int_{\zR^n} |\tfi^{b,m}  f(x) |w(x)g(x)\,dx \nonumber \\
	&\quad \le\kappa\normadefp{b}{\mathcal{L}^1_a}^m
	\sum_{(k,j)\in\zZ\times\zZ}
	|Q_{k,j}|
	\frac{\normadefp{\car_{\gamma Q_{k,j}} fv}{\mu}}
	{\normadefp{\car_{\gamma Q_{k,j}}}{\mu}}
	\frac{\normadefp{\car_{\gamma Q_{k,j}} g}{\nu}}
	{\normadefp{\car_{\gamma Q_{k,j}}}{\nu}}
	\frac{\normadeflp{\car_{\gamma Q_{k,j}}}{p}}
	{\normadeflp{\car_{\gamma Q_{k,j}}}{q}}.
	\end{align*}
	Let $\beta(\cdot)$ defined as in Lemma \ref{equivalenciabeta}.
	Then, by this lemma, the last sum is equivalent to
	\begin{align}\label{last sum}
	\kappa\normadefp{b}{\mathcal{L}^1_a}^m
	\sum_{(k,j)\in\zZ\times\zZ}
	|Q_{k,j}|
	\frac{\normadefp{\car_{\gamma Q_{k,j}} fv}{\mu}}
	{\normadefp{\car_{\gamma Q_{k,j}}}{\mu}}
	\normadeflp{\car_{\gamma Q_{k,j}}}{\beta}
	\frac{\normadefp{\car_{\gamma Q_{k,j}} g}{\nu}}
	{\normadefp{\car_{\gamma Q_{k,j}}}{\nu}}.
	\end{align}        
	For each $k,j\in\zZ$ we can consider the sets
	$D_k=\bigcup_{j\in\zZ}
	Q_{k,j}$ and $F_{k,j}=Q_{k,j}\backslash(Q_{k,j}\cap D_{k+1})$.
	Thus $\{F_{k,j}\}_{(k,j)\in\zZ\times\zZ}$ is a disjoint family of sets which satisfy
	\begin{equation}\label{A pelito}
	|Q_{k,j}\cap D_{k+1}|<\frac{\Pi}{\alpha}|Q_{k,j}|
	\end{equation}
	for some positive constant $\Pi<\alpha$, and
	\begin{equation}\label{B pelito}
	|Q_{k,j}|<\frac{1}{1-{\Pi}/{\alpha}}|F_{k,j}|.
	\end{equation}
	Deferring the proof of these inequalities for the moment, we can estimate \eqref{last sum} to obtain
	\begin{align*}
	&\int_{\z} |\tfi^{b,m}  f(x) |w(x)g(x)\,dx\\
	&\quad\lesssim \kappa
	\normadefp{b}{\mathcal{L}^1_a}^m
	\sum_{(k,j)\in\zZ\times\zZ}
	|F_{k,j}|
	\frac{\normadefp{\car_{\gamma Q_{k,j}} fv}{\mu}}
	{\normadefp{\car_{\gamma Q_{k,j}}}{\mu}}
	\normadeflp{\car_{\gamma Q_{k,j}}}{\beta}
	\frac{\normadefp{\car_{\gamma Q_{k,j}} g}{\nu}}
	{\normadefp{\car_{\gamma Q_{k,j}}}{\nu}}\\
	&\quad\lesssim \kappa\normadefp{b}{\mathcal{L}^1_a}^m
	\int_{\z}
	M_{L^\mu}(fv)(y)dy
	\mbetasf{\beta}{\nu}{(g)}(y)\\
	&\quad\lesssim \kappa\normadefp{b}{\mathcal{L}^1_a}^m
	\normadeflp{M_{L^\mu}(fv)}{p}
	\normadeflp{\mbetasf{\beta}{L^\nu}{(g)}}{p'}\\
	&\quad\lesssim \kappa\normadefp{b}{\mathcal{L}^1_a}^m
	\normadeflp{fv}{p},
	\end{align*}
	where we have used that by Theorem \ref{teo logl}, $M_{L^\mu}:\lpr{p}\hookrightarrow\lpr{p}$ since $p^->\mu$, and by Remark \ref{teo max llogl remark}, $\mbetas{\beta}{L^\nu}:\lpr{q'}\hookrightarrow\lpr{p'}$ since $(q')^->\nu$ (see \eqref{Mvar} and \eqref{Mvarfrac} for the definition of this maximal operatos).
	
	To prove \eqref{A pelito}, note that if for some $k,j,i\in\zZ$, $Q_{k,j}\cap Q_{k+1,i}\neq\emptyset$ then, by maximality and the fact that $\alpha> 1$, $Q_{k+1,i}\subsetneq Q_{k,j}$.
	Thus 
	\begin{align*}
	|Q_{k,j}\cap D_{k+1}|
	&=\left|Q_{k,j}\cap \bigcup_{i\in\zZ} Q_{k+1,i}\right|
	=\left|\bigcup_{i\in\zZ} (Q_{k,j}\cap  Q_{k+1,i})\right|
	=\sum_{i:Q_{k+1,i}\subseteq Q_{k,j}}|Q_{k+1,i}|\\
	&\le C_\tau^* \sum_{i:Q_{k+1,i}\subseteq
		Q_{k,j}}\normadeflp{\car_{Q_{k+1,i}}}{\tau}	\normadeflp{\car_{Q_{k+1,i}}}{\tau'}
	\end{align*}
	where the constant $C_{\tau}^*$ is provided by Lemma \ref{p en plog}.
	On the other hand, by maximality and the property \eqref{114} of the cubes $Q_{k+1,i}$ and $Q_{k,j}$ we have
	\begin{equation}\label{423}
	{\rm (i)}\,\, \alpha^{k+1}<\frac{\normadeflp{\car_{Q_{k+1,i}}gw}{\tau}}{\normadeflp{\car_{Q_{k+1,i}}}{\tau}}
	\,\,\,\,\text{and}\,\,\,\,{\rm (ii)}\,\,
	\frac{\normadeflp{\car_{Q_{k,j}}gw}{\tau}}{\normadeflp{\car_{Q_{k,j}}}{\tau}}\leq C_\tau \alpha^k
	\end{equation}	
	Then, by \eqref{423}{(i)} we have 
	\begin{align}\label{244}
	|Q_{k,j}\cap D_{k+1}|
	&\le C_\tau^* \sum_{i:Q_{k+1,i}\subseteq Q_{k,j}}\normadeflp{\car_{Q_{k+1,i}}}{\tau}	\normadeflp{\car_{Q_{k+1,i}}}{\tau'}\nonumber\\
	&< C_{\tau}^*\alpha^{-(k+1)} \sum_{i:Q_{k+1,i}\subseteq Q_{k,j}}	\normadeflp{\car_{Q_{k+1,i}}gw\car_{Q_{k,j}}}{\tau}	\normadeflp{\car_{Q_{k+1,i}}\car_{Q_{k,j}}}{\tau'}.
	\end{align}
	Note that, by Theorem \ref{7.3.22}, the following inequality holds 
	$$\sum_{i\in\zZ}\normadeflp{\car_{Q_{k+1,i}} r}{\tau}\normadeflp{\car_{Q_{k+1,i}} h}{\tau'}\le G_\tau \normadeflp{r}{\tau}\normadeflp{h}{\tau'}$$
	for every $r\in\lpr{\tau}$ and  $h\in\lpr{\tau'}$. 
	Appliying this with $r:=gw\car_{Q_{k,j}}$ and $h:=\car_{Q_{k,j}}$ we can estimate \eqref{244} as follows
	\begin{align*}
	|Q_{k,j}\cap D_{k+1}|
	&<C_{\tau}^*\alpha^{-(k+1)} G_\tau	\normadeflp{gw\car_{Q_{k,j}}}{\tau}	\normadeflp{\car_{Q_{k,j}}}{\tau'}.
	\end{align*}
	Then, by \eqref{423}(ii), we obtain that 
	\begin{align*}
	|Q_{k,j}\cap D_{k+1}|
	&<C_{\tau}^*\alpha^{-(k+1)}C_\tau \alpha^k G_\tau	\normadeflp{\car_{Q_{k,j}}}{\tau}	\normadeflp{\car_{Q_{k,j}}}{\tau'}\\
	&\leq C_{\tau}^*\alpha^{-(k+1)}C_\tau \alpha^k G_\tau C_\tau^{**}	|Q_{k,j}|:= \frac{\Pi}{\alpha}|Q_{k,j}|
	\end{align*}
	where the constant $C_{\tau}^{**}$ is provided by Lemma \ref{p en plog}. This gives \eqref{A pelito}. 
	Finally,
	\begin{align*}
	\frac{|F_{k,j}|}{|Q_{k,j}|}	
	&=\frac{|Q_{k,j}\setminus(Q_{k,j}\cap D_{k+1})|}{|Q_{k,j}|}
	=1-\frac{|Q_{k,j}\cap D_{k+1}|}{|Q_{k,j}|}>1-\frac{\Pi}{\alpha}>0
	\end{align*}
	since, $\alpha>\Pi$, and we obtain \eqref{B pelito}. 
	
	In order to complete the proof we must show that (\ref{116}) holds.
	In fact, if $\ell(Q_0)=2^{-d_0}$ with $d_0\in\zZ$, by Lemma \ref{p en plog} we have
	\begin{align*}
	&\sum_{Q\,:\, Q\subset Q_0}
	\overline{K}\left(\frac{\ell(Q)}{2}\right)
	|3Q||Q|
	\frac{\normadeflp{\car_{3Q} f}{\omega}}{\normadeflp{\car_{3Q}}{\omega}}\\
	&\quad\quad\lesssim\sum_{d\geq d_0}\overline{K}(2^{-d-1})
	2^{-d n}\sum_{Q\subset Q_0\,:\,\ell(Q)=2^{-d }}
	\normadeflp{f\car_{3Q}}{\omega}\normadeflp{\car_{3Q}}{\omega'}.
	\end{align*}
	Thus, applying Lemma \ref{lemaQ0} with $f$ and $g:=\car_{{3Q_0}}$, we obtain that 
	\begin{align*}
	\sum_{Q\,:\, Q\subset Q_0}
	\overline{K}\left(\frac{\ell(Q)}{2}\right)
	|3Q||Q|
	\frac{\normadeflp{\car_{3Q} f}{\omega}}{\normadeflp{\car_{3Q}}{\omega}}&\lesssim
	\normadeflp{f\car_{3Q_0}}{\omega}\normadeflp{\car_{3Q_0}}{\omega'}\sum_{d\geq d_0}\overline{K}(2^{-d-1})                2^{-d n}\\
	&\lesssim
	\normadeflp{f\car_{3Q_0}}{\omega}\normadeflp{\car_{3Q_0}}{\omega'}\widetilde{K}(\delta(1+\varepsilon)\ell(Q_0)),
	\end{align*}
	where the last estimate follows as in \cite{P}.
	This proves \eqref{116} and concludes the proof of Theorem \ref{teo conm variable con lipschitz generalizado}.
\end{proof}
\medskip
\begin{proof}[Proof of Theorem~\ref{teo general conmut lipchitz constante}]
	We use the same technique as in the proof of the Theorem \ref{teo conm variable con lipschitz generalizado} to obtain that
	\begin{align*}
	&\int_{\zR^n} |\tfi^{b,m}  f(x) |w(x)g(x)\,dx \nonumber\\
	&\quad\lesssim\sum_Q \overline{K}\left(\frac{\ell(Q)}{2}\right)
	\sum_{j=0}^m
	\int_{3Q}|b(z)-b_Q|^{j}f(z)\,dz
	\,\int_Q |b(x)-b_Q|^{m-j}g(x)w(x)\,dx.
	\end{align*}
	Hence, by Lemma \ref{obs teorema delta variable},
	\begin{align}\label{151}
	&\int_{\z} |\tfi^{b,m}  f(x) |w(x)g(x)\,dx \nonumber\\
	&\quad\lesssim\sum_Q \overline{K}\left(\frac{\ell(Q)}{2}\right)
	\normadefp{\car_{Q}}{n/\delta(\cdot)}^m |Q|
	\int_{3Q}f(z)\,dz
	\,\fint_Q g(x)w(x)\,dx
	\end{align}
	Thus, given some constant $\alpha$ larger than $2^n$ and proceeding as in [\cite{P}, Proof of Theorem 2.1], for each $k\in\zZ$ there exists a family of maximal non-overlaping dyadic cubes $\{Q_{k,j}\}_{j\in\zZ}$, the  Calderón-Zygmund cubes, such that we can estimate (\ref{151}) by a multiple of
	\begin{align}\label{152}
	\sum_{(k,j)\in\zZ\times\zZ}
	\widetilde{K}(\ell(\gamma Q_{k,j}))
	\normadefp{\car_{Q_{k,j}}}{n/\delta(\cdot)}^m |Q_{k,j}|
	\fint_{\gamma Q_{k,j}}f(z)dz \,
	\fint_{\gamma Q_{k,j}} g(z)w(z)dz,
	\end{align}
	where $\gamma=\max \{3,\delta(1+\varepsilon)\}$ with $\varepsilon,\delta$ the numbers provided by condition $\mathfrak{D}$.
	By condition $\mathcal{F}$ and H\"{o}lder's inequality we have
	\begin{align*}
	\fint_{\gamma Q_{k,j}}f(z) \,dz
	\lesssim
	\frac{\normadeflpl{\car_{\gamma Q_{k,j}} f}{D}}
	{\normadeflpl{\car_{\gamma Q_{k,j}}}{D}}
	\frac{\normadeflpl{\car_{\gamma Q_{k,j}}}{D^*}}
	{\normadeflpl{\car_{\gamma Q_{k,j}}}{D^*}}
	\lesssim
	\frac{\normadeflpl{\car_{\gamma Q_{k,j}} fv}{B}}
	{\normadeflpl{\car_{\gamma Q_{k,j}}}{B}}
	\frac{\normadeflpl{\car_{\gamma Q_{k,j}} v^{-1}}{A}}
	{\normadeflpl{\car_{\gamma Q_{k,j}}}{A}}
	\end{align*}
	and
	\begin{align*}
	\fint_{\gamma Q_{k,j}}g(z)w(z) \,dz
	\lesssim
	\frac{\normadeflpl{\car_{\gamma Q_{k,j}} gw}{J}}
	{\normadeflpl{\car_{\gamma Q_{k,j}}}{J}}
	\frac{\normadeflpl{\car_{\gamma Q_{k,j}}}{J^*}}
	{\normadeflpl{\car_{\gamma Q_{k,j}}}{J^*}}
	\lesssim
	\frac{\normadeflpl{\car_{\gamma Q_{k,j}} g}{H}}
	{\normadeflpl{\car_{\gamma Q_{k,j}}}{H}}
	\frac{\normadeflpl{\car_{\gamma Q_{k,j}} w}{E}}
	{\normadeflpl{\car_{\gamma Q_{k,j}}}{E}}.
	\end{align*}
	Then from \eqref{152} and by Fefferman-Phong type condition \eqref{hipotesis FM MO} on the weights we have
	\begin{align*}
	& \int_{\z} |\tfi^{b,m}  f(x) |w(x)g(x)\,dx \nonumber\\
	&\quad\lesssim
	\sum_{(k,j)\in\zZ\times\zZ}
	\widetilde{K}(l(\gamma Q_{k,j}))
	\normadefp{\car_{Q_{k,j}}}{n/\ex{\delta}}^m |Q_{k,j}|
	\frac{\normadeflpl{\car_{\gamma Q_{k,j}}fv}{B}}
	{\normadeflpl{\car_{\gamma Q_{k,j}}}{B}}
	\frac{\normadeflpl{\car_{\gamma Q_{k,j}}v^{-1}}{A}}
	{\normadeflpl{\car_{\gamma Q_{k,j}}}{A}}\\
	&\quad\quad\quad\times
	\frac{\normadeflpl{\car_{\gamma Q_{k,j}}g}{H}}
	{\normadeflpl{\car_{\gamma Q_{k,j}}}{H}}
	\frac{\normadeflpl{\car_{\gamma Q_{k,j}}w}{E}}
	{\normadeflpl{\car_{\gamma Q_{k,j}}}{E}}\\
	&\quad\le \kappa
	\sum_{(k,j)\in\zZ\times\zZ}
	|Q_{k,j}|
	\frac{\normadeflpl{\car_{\gamma Q_{k,j}}fv}{B}}
	{\normadeflpl{\car_{\gamma Q_{k,j}}}{B}}
	\frac{\normadeflpl{\car_{\gamma Q_{k,j}}g}{H}}
	{\normadeflpl{\car_{\gamma Q_{k,j}}}{H}}
	\frac{\normadeflp{\car_{\gamma Q_{k,j}}}{p}}{\normadeflp{\car_{\gamma Q_{k,j}}}{q}}.
	\end{align*}
	Let $\beta(\cdot)$ be defined as in Lemma \ref{equivalenciabeta}.
	Then, by this lemma, the last sum is equivalent to
	\begin{equation*}
	\kappa\sum_{(k,j)\in\zZ\times\zZ}
	|Q_{k,j}|
	\frac{\normadeflpl{\car_{\gamma Q_{k,j}}fv}{B}}
	{\normadeflpl{\car_{\gamma Q_{k,j}}}{B}}
	\normadeflp{\car_{\gamma Q_{k,j}}}{\beta}
	\frac{\normadeflpl{\car_{\gamma Q_{k,j}}g}{H}}
	{\normadeflpl{\car_{\gamma Q_{k,j}}}{H}}.
	\end{equation*}
	We shall use the following properties of Calderón-Zygmund cubes.
	For each $k,j\in\zZ$ we can consider the sets $D_k=\bigcup_{j\in\zZ}Q_{k,j}$ and $F_{k,j}=Q_{k,j}\setminus(Q_{k,j}\cap D_{k+1})$.
	Thus $\{F_{k,j}\}_{(k,j)\in\zZ\times\zZ}$ is a disjoint family of sets  which satisfy 	
	$$|Q_{k,j}|<\frac{1}{1-\frac{2^n}{\alpha}}|F_{k,j}|.$$
	Then
	\begin{align*}
	\int_{\z} \tfi f(x) g(x)w(x)\,dx
	&\lesssim \kappa\sum_{(k,j)\in\zZ\times\zZ}
	|F_{k,j}|
	\frac{\normadeflpl{\car_{\gamma Q_{k,j}}fv}{B}}
	{\normadeflpl{\car_{\gamma Q_{k,j}}}{B}}
	\normadeflp{\car_{\gamma Q_{k,j}}}{\beta}
	\frac{\normadeflpl{\car_{\gamma Q_{k,j}}g}{H}}
	{\normadeflpl{\car_{\gamma Q_{k,j}}}{H}}\\
	&\leq  \kappa\int_{\zR^n} M_{B(L,\cdot)}(fv)(y)
	\mbetasf{\beta}{H(L,\cdot)}{(g)}(y)dy\\
	&\lesssim\kappa                    \normadeflp{M_{B(L,\cdot)}(fv)}{p}
	\normadeflp{\mbetasf{\beta}{H(L,\cdot)}{(g)}}{p'}\\
	&\lesssim \kappa\normadeflp{fv}{p}
	\end{align*}
	where we have used the hyphotesis (\ref{hipotesis maximal}) and (\ref{hipotesis maximal fraccionaria}).
\end{proof}

\bibliographystyle{plain}
\bibliography{BIBLIOGRAFIALU2}

\end{document}